\documentclass[12pt,a4paper]{amsart}
\usepackage{amsmath,amssymb,a4wide,amsthm,ifthen,hyperref,xcolor,enumerate,enumitem}
\usepackage[utf8]{inputenc}

\usepackage{graphicx}
\usepackage{url}
\usepackage{ulem}
\usepackage{epstopdf}
\usepackage[section]{placeins}
\usepackage{natbib}
\usepackage{flafter} 
\usepackage{multirow}

\newtheorem{theo}{Theorem}
\newtheorem{prop}{Proposition}
\newtheorem{lemma}{Lemma}
\theoremstyle{remark}

\begin{document}

\title[Variable selection in sparse high-dimensional GLARMA models]{Variable selection in sparse high-dimensional GLARMA models}

\date{\today}

\author{C. L\'evy-Leduc}
\address{UMR MIA-Paris, AgroParisTech, INRA, Universit\'e Paris-Saclay, 75005, Paris, France}
\email{celine.levy-leduc@agroparistech.fr}
\author{S. Ouadah}
\address{UMR MIA-Paris, AgroParisTech, INRA, Universit\'e  Paris-Saclay, 75005, Paris, France}
\email{sarah.ouadah@agroparistech.fr}
\author{L. Sansonnet}
\address{UMR MIA-Paris, AgroParisTech, INRA, Universit\'e Paris-Saclay, 75005, Paris, France}
\email{laure.sansonnet@agroparistech.fr}


\maketitle

\begin{abstract}
  In this paper, we propose a novel variable selection approach in the framework of sparse high-dimensional GLARMA models.
  It consists in combining the estimation of the autoregressive moving average (ARMA) coefficients of these models with regularized methods designed for Generalized Linear Models (GLM). 
 The properties of our approach are investigated both from a theoretical and a numerical point of view. More
precisely, we establish \textcolor{black}{in a specific case} the consistency of the ARMA part coefficient estimators. We explain how to implement our approach and we show that it is
very attractive since it benefits from a low computational load. We also assess the performance of
our methodology using synthetic data and compare it with alternative approaches. Our numerical
experiments show that combining the estimation of the ARMA part coefficients with regularized methods designed for GLM
dramatically improves the variable selection performance.
\end{abstract}

\section{Introduction}

\textcolor{black}{Discrete-valued time series arise in a wide variety of fields ranging from finance to molecular biology and public health.
For instance, we can mention the number of transactions in stocks in the finance field, see \cite{brannas:quoreshi:2010}.
In the field of molecular biology, modeling RNA-Seq kinetics data is a challenging issue, see \cite{Thorne:2018} and in the public health
context, there is an interest in the modeling of daily asthma presentations in a given hospital, see \cite{SOUZA:2014}.}

% Modeling discrete-valued time series is a major
% statistical issue that arises in many fields such as
% .... 
% \begin{itemize}
% \item epidemiology/medecine : count of cases of a certain disease, asthma-pollution
% \item economics/telecommunications : count of price changes
% \item public health: the daily number of hospital admissions
% \item molecular biology : RNA-seq kinetics data
% \item finance: the number of stock market transactions per minute
% \item industrial quality control: the hourly number of defect items
% \end{itemize}
% parler des Many covariates are often available  .... ?.
% \bigskip

The literature on modeling discrete-valued time series is becoming increasingly abundant, see \cite{handbook:2016} for a review. 
Different classes of models have been proposed such as the Integer Autoregressive Moving Average 
(INARMA) models and the generalized state space models. 

The Integer Autoregressive process of order 1 (INAR(1)) was first introduced by \cite{McKenzie:1985} and the Integer-valued Moving Average 
(INMA) process is described in \cite{Al-Osh:1988}. One of the attractive features of INARMA processes is that their autocorrelation structure is similar 
to the one of autoregressive moving average (ARMA) models. However, it has to be noticed that statistical inference in these models is generally complicated and requires to develop 
intensive computational approaches such as the 
efficient MCMC algorithm devised by \cite{Neal:rao:2007} for INARMA processes of known AR and MA orders. This strategy was extended 
to unknown AR and MA orders by \cite{enciso:nea:rao:2009}.
 For further references on INARMA models, we refer the reader to \cite{weiss:dts}. 

%\textcolor{black}{However, a major drawback is that they cannot produce negative correlations.}

%There exist different classes of models with their specific meaning and estimation method. \bigskip

% An important class of models for discrete-valued time series is based on thinning operators, like the integer autoregressive moving average (INARMA) models.
% %The integer-valued autoregressive (INAR) process make use of thinning operations to obtain an ARMA-like autocorrelation structure. 
% The INAR(1) process was first introduced by McKenzie (1985) (see also Al-Osh and Alzaid (1987) and Weiss (le livre)) and the literature dedicated to these processes mainly focus on the PINAR(1) process (see, e.g. Bourguignon and Weiss Test, 2017; Bourguignon et al. JoAS, 2018), meaning an INAR process modeling a time series with a lag of one and a Poisson distribution. Having the good property to be constructed such that their autocorrelation structure is similar to the one of autoregressive moving average (ARMA) models, these models all have the drawback that they cannot produce negative correlations.
% \bigskip

% Statistical
% inference for INAR models is generally more complicated than for the Poisson
% regression models and most attention has been devoted to the INAR(1)
% model. However, recently in Neal and Subba Rao (2007) an efficient MCMC
% algorithm has been developed for INARMA processes of known AR and MA
% orders with extensions to unknown AR and MA orders given in Enciso-Mora
% et al. (2009).

The other important class of models for discrete-valued time series is the one of generalized state space models which can have a parameter-driven and an observation-driven version, 
see \cite{davis:1999} for a review. 

The main difference between \textcolor{black}{these two versions} is that in parameter-driven models, the state vector evolves independently of the past history
of the observations whereas the state vector depends on the past observations in observation-driven models. More precisely, in parameter-driven models, let $(\nu_t)$ be a stationary process,
the observations $Y_t$ are thus modeled as follows: conditionally on $(\nu_t)$, $Y_t$ has a Poisson distribution of parameter $\exp(\beta_0^\star+\sum_{i=1}^p\beta_i^\star x_{t,i}+\nu_t)$,
where the $x_{t,i}$'s are the $p$ regressor variables (or covariates). Estimating the parameters in such models has a very high computational load, see \cite{jung:2001}.

Observation-driven models initially proposed by \cite{cox:1981} and further studied in \cite{zeger:qaqish:1988} do not have this computational drawback and are thus considered as a promising alternative to parameter-driven models.  Different kinds of observation-driven models can be found in the literature: 
the Generalized Linear Autoregressive Moving Average (GLARMA) models introduced by \cite{davis:1999}
 and further studied in \cite{davis:dunsmuir:streett:2003,davis:dunsmuir:street:2005,dunsmuir:2015} and the (log-)linear Poisson autoregressive models introduced in 
\cite{fokianos:2009,fokianos:2011,fokianos:2012}. Note that GLARMA models cannot be seen as a particular case of the log-linear Poisson autoregressive models.

\textcolor{black}{In the following, we shall consider the GLARMA model introduced in \cite{davis:dunsmuir:street:2005} with additional covariates. More precisely,} 
given the past history $\mathcal{F}_{t-1}=\sigma(Y_s,s\leq t-1)$, \textcolor{black}{we assume that}
\begin{equation}\label{eq:Yt}
Y_t|\mathcal{F}_{t-1}\sim\mathcal{P}\left(\mu_t^\star\right),
\end{equation}
where $\mathcal{P}(\mu)$ denotes the Poisson distribution with mean $\mu$. In (\ref{eq:Yt}),
\begin{equation}\label{eq:mut_Wt}
\mu_t^\star=\exp(W_t^\star) \textrm{ with } W_t^\star=\beta_0^\star+\sum_{i=1}^p\beta_i^\star x_{t,i}+Z_t^\star,
\end{equation}
where the $x_{t,i}$'s are the $p$ regressor variables ($p\geq 1$),
\begin{equation}\label{eq:Zt}
Z_t^\star=\sum_{j=1}^q \gamma_j^\star E_{t-j}^\star \textrm{ with } E_t^\star=\frac{Y_t-\mu_t^\star}{\mu_t^\star}=Y_t\exp(-W_t^\star)-1,
\end{equation}
where $1\leq q\leq\infty$
and $E_t^\star=0$ for all $t\leq 0$. \textcolor{black}{Here, the $E_t^\star$'s correspond to the working residuals in classical Generalized Linear Models (GLM), which means that we limit ourselves to the case 
$\lambda=1$ in the more general definition:
$
 E_t^\star=(Y_t-\mu_t^\star){\mu_t^{\star}}^{-\lambda}
$. Note that in the case where $q=\infty$, $(Z_t^\star)$ satisfies the ARMA-like recursions given in Equation (4) of \cite{davis:dunsmuir:street:2005}.
The model defined by (\ref{eq:Yt}), (\ref{eq:mut_Wt}) and (\ref{eq:Zt}) is thus referred as a GLARMA model.}

The main goal of this paper is to introduce a novel variable selection approach in the deterministic part \textcolor{black}{(covariates)} of high-dimensional sparse 
GLARMA models that is in (\ref{eq:Yt}) and (\ref{eq:mut_Wt})
where $p$ is large and when the vector of the $\beta_i$'s is sparse. \textcolor{black}{The novel approach that we propose consists in} combining a procedure 
for estimating the ARMA part coefficients with regularized methods designed for GLM.

The paper is organized as follows. We firstly describe an estimation procedure for the ARMA part of the GLARMA model defined in (\ref{eq:Yt}), (\ref{eq:mut_Wt}) and (\ref{eq:Zt}), 
see Section \ref{sec:estim}, 
and establish its consistency in \textcolor{black}{a specific case, see Section \ref{sec:consistency}.}
Secondly, we propose a novel variable selection approach in the regression part of the sparse high-dimensional model (\ref{eq:Yt}) and explain how to combine it with 
the estimation procedure of the ARMA part coefficients, see Section \ref{sec:variable}. \textcolor{black}{The practical implementation of our approach is given in Section \ref{sec:practical}.}
Thirdly, in Section \ref{sec:num}, some numerical experiments are provided to illustrate our method and to compare its performance to alternative approaches
on finite sample size data. The proofs of the theoretical results are given in Section \ref{sec:proofs}.

%%% Local Variables:
%%% mode: latex
%%% eval: (TeX-PDF-mode 1)
%%% TeX-master: "papier_glarma_jspi.tex"
%%% ispell-local-dictionary: "en_US"
%%% eval: (flyspell-mode 1)
%%% End:

\section{Statistical inference}\label{sec:stat_inf}

\subsection{Estimation procedure}\label{sec:estim}

% In order to estimate the parameter $\boldsymbol{\delta}^\star=(\boldsymbol{\beta}^\star,\boldsymbol{\gamma}^\star)$
% where $\boldsymbol{\beta}^\star=(\beta_0^\star,\beta_1^\star,\dots,\beta_p^\star)$ is the vector of regressor coefficients defined in (\ref{eq:mut_Wt})
% and $\boldsymbol{\gamma}^\star=(\gamma_1^\star,\dots,\gamma_q^\star)$ is the vector of the MA part coefficients defined in (\ref{eq:Zt}), we maximize the following criterion, 
% based on the conditional log-likelihood, with respect to $\boldsymbol{\delta}=(\boldsymbol{\beta},\boldsymbol{\gamma})$:

%\textcolor{blue}{Je pense qu'on doit passer par les transposées (voir section 2.2) ... Et est-ce que l'on note la transposée $'$ ? \\
For estimating the parameter $\boldsymbol{\delta}^\star=(\boldsymbol{\beta}^{\star\prime},\boldsymbol{\gamma}^{\star\prime})$
where $\boldsymbol{\beta}^\star=(\beta_0^\star,\beta_1^\star,\dots,\beta_p^\star)'$ is the vector of regressor coefficients defined in (\ref{eq:mut_Wt})
and $\boldsymbol{\gamma}^\star=(\gamma_1^\star,\dots,\gamma_q^\star)'$ is the vector of the ARMA part coefficients defined in (\ref{eq:Zt}), we maximize the following criterion,
based on the conditional log-likelihood, with respect to $\boldsymbol{\delta}=(\boldsymbol{\beta}',\boldsymbol{\gamma}')$, with
$\boldsymbol{\beta}=(\beta_0,\beta_1,\dots,\beta_p)'$ and $\boldsymbol{\gamma}=(\gamma_1,\dots,\gamma_q)'$:
\begin{equation}\label{eq:likelihood}
L(\boldsymbol{\delta})=\sum_{t=1}^n\left(Y_t W_t(\boldsymbol{\delta})-\exp(W_t(\boldsymbol{\delta}))\right),
\end{equation}
where 
% \begin{equation}\label{eq:Wt}
% W_t(\boldsymbol{\delta})=\boldsymbol{\beta}\; x_t'+Z_t(\boldsymbol{\delta})=\beta_0+\sum_{i=1}^p\beta_i x_{t,i}+\sum_{j=1}^q \gamma_j E_{t-j}(\boldsymbol{\delta}),\textrm{ with } 
% E_t(\boldsymbol{\delta})=Y_t\exp(-W_t(\boldsymbol{\delta}))-1,
% \end{equation}
% $x_t'$ denoting the transposition of $x_t=(x_{t,0},x_{t,1},\dots,x_{t,p})$ with $x_{t,0}=1$ for all $t$.
 %where 
\begin{equation}\label{eq:Wt}
W_t(\boldsymbol{\delta})=\boldsymbol{\beta}'x_t+Z_t(\boldsymbol{\delta})=\beta_0+\sum_{i=1}^p\beta_i x_{t,i}+\sum_{j=1}^q \gamma_j E_{t-j}(\boldsymbol{\delta}),
\end{equation}
with $x_t=(x_{t,0},x_{t,1},\dots,x_{t,p})'$, $x_{t,0}=1$ for all $t$ and

\begin{eqnarray}
E_t(\boldsymbol{\delta})=Y_t\exp(-W_t(\boldsymbol{\delta}))-1,\mbox{ if }t>0\mbox{ and }E_t(\boldsymbol{\delta})=0\mbox{, if }t\leq 0.
\label{eq:Et}
\end{eqnarray}
%
%with $x_t=(x_{t,0},x_{t,1},\dots,x_{t,p})'$ with $x_{t,0}=1$ for all $t$, $E_t(\boldsymbol{\delta})=Y_t\exp(-W_t(\boldsymbol{\delta}))-1$, if $t>0$ and $E_t(\boldsymbol{\delta})=0$, if $t\leq 0$.
For further details on the choice of this criterion, we refer the reader to \cite{davis:dunsmuir:street:2005}.
To obtain $\widehat{\boldsymbol{\delta}}$ defined by
\begin{equation}\label{eq:delta_hat}
\widehat{\boldsymbol{\delta}}=\textrm{Argmax}_{\boldsymbol{\delta}} \; L(\boldsymbol{\delta}),
\end{equation}
we consider the first derivatives of $L$:
\begin{equation*}
\frac{\partial L}{\partial \boldsymbol{\delta}}(\boldsymbol{\delta})=\sum_{t=1}^n(Y_t-\exp(W_t(\boldsymbol{\delta}))\frac{\partial W_t}{\partial \boldsymbol{\delta}}(\boldsymbol{\delta}),
\end{equation*}
where 
% \begin{equation*}
% \frac{\partial W_t}{\partial \boldsymbol{\delta}}(\boldsymbol{\delta})=\frac{\partial\boldsymbol{\beta} x_t'}{\partial \boldsymbol{\delta}}+\frac{\partial Z_t}{\partial \boldsymbol{\delta}}
% (\boldsymbol{\delta}),
% \end{equation*}
%\textcolor{blue}{
\begin{equation*}
\frac{\partial W_t}{\partial \boldsymbol{\delta}}(\boldsymbol{\delta})=\frac{\partial\boldsymbol{\beta}' x_t}{\partial \boldsymbol{\delta}}+\frac{\partial Z_t}{\partial \boldsymbol{\delta}}
(\boldsymbol{\delta}),
\end{equation*}
%}
$\boldsymbol{\beta}$, $x_t$ and $Z_t$ being defined in (\ref{eq:Wt}). 
\textcolor{black}{The computations of the first derivatives of $W_t$ are detailed in Section \ref{subsub:first_derive}. }

Since the first derivatives of $W_t$ are recursively defined,  it is not possible to obtain a closed-form formula for $\widehat{\boldsymbol{\delta}}$. 
Thus, in order to compute $\widehat{\boldsymbol{\delta}}$, we shall use
the following Newton-Raphson algorithm. 
% In order to compute $\widehat{\boldsymbol{\delta}}$ defined in (\ref{eq:delta_hat}) using the Newton-Raphson algorithm,
More precisely, we start from an initial value for $\boldsymbol{\delta}$ denoted 
$\boldsymbol{\delta}^{(0)}$. Then, we use the following recursion for $r\geq 1$: 
\begin{equation}\label{eq:newton_raphson}
\boldsymbol{\delta}^{(r)}=\boldsymbol{\delta}^{(r-1)}-\frac{\partial^2 L}{\partial \boldsymbol{\delta}'\partial \boldsymbol{\delta}}(\boldsymbol{\delta}^{(r-1)})^{-1}\frac{\partial L}{\partial \boldsymbol{\delta}}(\boldsymbol{\delta}^{(r-1)}),
\end{equation}
\textcolor{black}{where $\frac{\partial^2 L}{\partial \boldsymbol{\delta}'\partial \boldsymbol{\delta}}$ corresponds to the Hessian matrix of $L$ and is defined in (\ref{eq:def:hess}).}
Hence, it requires the computation of the first and second derivatives of $L$. 
We already explained how to compute the first derivatives of $L$. As for the second derivatives of $L$, it can be obtained as follows:

\begin{equation}\label{eq:def:hess}
\frac{\partial^2 L}{\partial \boldsymbol{\delta}'\partial \boldsymbol{\delta}}(\boldsymbol{\delta})
=\sum_{t=1}^n(Y_t-\exp(W_t(\boldsymbol{\delta}))\frac{\partial^2 W_t}{\partial \boldsymbol{\delta}'\partial\boldsymbol{\delta}}(\boldsymbol{\delta})
-\sum_{t=1}^n\exp(W_t(\boldsymbol{\delta}))\frac{\partial W_t}{\partial \boldsymbol{\delta}'}(\boldsymbol{\delta})\frac{\partial W_t}{\partial \boldsymbol{\delta}}(\boldsymbol{\delta}).
\end{equation}
\textcolor{black}{The computations of the second derivatives of $W_t$ are detailed in Section \ref{subsub:second_derive}. }

Further details on the choice of $\boldsymbol{\delta}^{(0)}$ and the number of iterations to use will be given in Section \ref{sec:num}.

\subsection{Variable selection}\label{sec:variable}

To perform variable selection in the $\beta_i^\star$ of Model (\ref{eq:mut_Wt}) that is to obtain a sparse estimator of $\beta_i^\star$, 
we shall use a methodology inspired by \cite{friedman:hastie:tibshirani:2010} 
for fitting generalized linear models with $\ell_1$ penalties. It consists in penalizing a quadratic approximation to the log-likelihood obtained by a Taylor expansion.
Hence, denoting $\widetilde{\boldsymbol{\beta}}=(\widetilde{\beta}_0,\dots,\widetilde{\beta}_p)'$ the current estimate of the parameter $\boldsymbol{\beta}^\star=({\beta_0}^\star,\dots,{\beta_p}^\star)'$, 
we obtain the following quadratic approximation where $\widehat{\boldsymbol{\gamma}}=(\widehat{\gamma}_1,\dots,\widehat{\gamma}_q)'$ is the estimate of 
$\boldsymbol{\gamma}^\star=({\gamma_1}^\star,\dots,{\gamma_q}^\star)'$ obtained in Section \ref{sec:estim}:
\begin{align*}
\widetilde{L}(\boldsymbol{\beta}):=L(\beta_0,\dots,\beta_p,\widehat{\gamma})=\widetilde{L}(\widetilde{\boldsymbol{\beta}})
+\frac{\partial L}{\partial \boldsymbol{\beta}}(\widetilde{\boldsymbol{\beta}},\widehat{\boldsymbol{\gamma}})(\boldsymbol{\beta}-\widetilde{\boldsymbol{\beta}})
+\frac12 (\boldsymbol{\beta}-\widetilde{\boldsymbol{\beta}})'
\frac{\partial^2 L}{\partial \boldsymbol{\beta}\partial \boldsymbol{\beta}'}(\widetilde{\boldsymbol{\beta}},\widehat{\boldsymbol{\gamma}})
(\boldsymbol{\beta}-\widetilde{\boldsymbol{\beta}}),
\end{align*}
where
$$\frac{\partial L}{\partial \boldsymbol{\beta}}=\left(\frac{\partial L}{\partial \beta_0},\dots,\frac{\partial L}{\partial \beta_p}\right)
\textrm{ and }
\frac{\partial^2 L}{\partial \boldsymbol{\beta}\partial \boldsymbol{\beta}'}=\left(\frac{\partial^2 L}{\partial \beta_j \partial \beta_k}\right)_{0\leq j,k\leq p}.$$
Thus,
\begin{align}\label{eq:Ltilde}
\widetilde{L}(\boldsymbol{\beta})=\widetilde{L}(\widetilde{\boldsymbol{\beta}})+\frac{\partial L}{\partial \boldsymbol{\beta}}(\widetilde{\boldsymbol{\beta}},\widehat{\boldsymbol{\gamma}})
U(\boldsymbol{\nu}-\widetilde{\boldsymbol{\nu}})-\frac12 (\boldsymbol{\nu}-\widetilde{\boldsymbol{\nu}})' \Lambda (\boldsymbol{\nu}-\widetilde{\boldsymbol{\nu}}),
\end{align}
%\textcolor{blue}{OK pour le "-" dans l'approximation de Taylor (à rediscuter) ; il faut aussi le mettre dans la définition de $\tilde{L}(\boldsymbol{\beta})$ ci-dessus}
where $U\Lambda U'$ is the singular value decomposition of the positive semidefinite symmetric matrix 
$-\frac{\partial^2 L}{\partial \boldsymbol{\beta}\partial \boldsymbol{\beta}'}(\widetilde{\boldsymbol{\beta}},\widehat{\boldsymbol{\gamma}})$
and $\boldsymbol{\nu}-\widetilde{\boldsymbol{\nu}}=U'(\boldsymbol{\beta}-\widetilde{\boldsymbol{\beta}})$.

In order to obtain a sparse estimator of $\boldsymbol{\beta}^\star$, we propose using $\widehat{\boldsymbol{\beta}}(\lambda)$ defined by
\begin{equation}\label{eq:beta_hat}
\widehat{\boldsymbol{\beta}}(\lambda)=\textrm{Argmin}_{\boldsymbol{\beta}}\left\{-\widetilde{L}_Q(\boldsymbol{\beta})+\lambda \|\boldsymbol{\beta}\|_1\right\},
\end{equation}
for a positive $\lambda$, where $\|\boldsymbol{\beta}\|_1=\sum_{k=0}^p |\beta_k|$ and $\widetilde{L}_Q(\boldsymbol{\beta})$ denotes the quadratic approximation of the log-likelihood. 
\textcolor{black}{This quadratic approximation is} defined by
\begin{equation}\label{eq:LQtilde}
-\widetilde{L}_Q(\boldsymbol{\beta})=\frac12\|\mathcal{Y}-\mathcal{X}\boldsymbol{\beta}\|_2^2,
\end{equation}
with
\begin{equation}\label{eq:def_Y_X}
\mathcal{Y}=\Lambda^{1/2}U'\widetilde{\boldsymbol{\beta}}
-\Lambda^{-1/2}U'\left(\frac{\partial L}{\partial \boldsymbol{\beta}}(\widetilde{\boldsymbol{\beta}},\widehat{\boldsymbol{\gamma}})\right)' ,\;  \mathcal{X}=\Lambda^{1/2}U'
\end{equation}
and $\|\cdot\|_2$ denoting the $\ell_2$ norm in $\mathbb{R}^{p+1}$.
\textcolor{black}{Computational details for obtaining the expression \eqref{eq:LQtilde} of $\widetilde{L}_Q(\boldsymbol{\beta})$ appearing in
Criterion (\ref{eq:beta_hat}) are provided in Section \ref{sub:var_sec}.}

\textcolor{black}{The parameter $\lambda$ is tuned thanks to the stability
selection approach devised by \cite{meinshausen:buhlmann:2010}. For further details, we refer the reader to Section \ref{sec:num}.
}

\subsection{\textcolor{black}{Practical implementation}}\label{sec:practical}
We summarize hereafter the different steps of our methodology.% which is implemented in the R package \texttt{HighDimGlarma}.
\begin{itemize}
\item\textsf{\underline{First step:} Initialization.} We take for $\boldsymbol{\beta}^{(0)}$ 
the estimator of $\boldsymbol{\beta}^\star$ obtained by fitting a GLM to the observations
$Y_1,\dots,Y_n$ thus ignoring the ARMA part of the model. For $\boldsymbol{\gamma}^{(0)}$, we take the null vector.
% use standard estimators of MA($q$) process
% fitted to $\left(\log(Y_t)-\boldsymbol{\beta}^{(0)}x_t'\right)_{1\leq t\leq n}$.
\item\textsf{\underline{Second step:} Newton-Raphson algorithm.} We use the recursion defined in (\ref{eq:newton_raphson}) with
the initialization $\boldsymbol{\delta}^{(0)}=(\boldsymbol{\beta}^{(0)},\boldsymbol{\gamma}^{(0)})$ obtained in the first step and
we stop at the iteration $R$ such that $\|\boldsymbol{\delta}^{(R)}-\boldsymbol{\delta}^{(R-1)}\|_\infty<10^{-6}$.
\item\textsf{\underline{Third step:} Variable selection.} To obtain a sparse estimator of $\boldsymbol{\beta}^\star$, we use the criterion (\ref{eq:beta_hat})
where $\widetilde{\beta}$ and $\widehat{\boldsymbol{\gamma}}$ appearing in (\ref{eq:def_Y_X}) are replaced by $\boldsymbol{\beta}^{(0)}$ and $\boldsymbol{\gamma}^{(R)}$ 
obtained in the first and second steps, respectively. This step provides
$\widehat{\boldsymbol{\beta}}(\lambda)$ for different values of $\lambda$.
\item\textsf{\underline{Fourth step:} Choice of $\lambda$.} To choose the value of $\lambda$ and thus the final estimator $\widehat{\boldsymbol{\beta}}$ of $\boldsymbol{\beta}^\star$, we use the stability
selection approach devised by \cite{meinshausen:buhlmann:2010}.
% \item\textsf{\underline{Fifth step:}} Replace $\boldsymbol{\beta}^{(0)}$ in the first step by $\widehat{\boldsymbol{\beta}}$ obtained in the third step to obtain
% a new $\boldsymbol{\gamma}^{(0)}$ and repeat the second and third steps.
\end{itemize}

\subsection{Consistency results}\label{sec:consistency}

In this section, we shall establish \textcolor{black}{in the case where $q=1$} 
the consistency of the parameter $\gamma_1^\star$ from $Y_1,\dots,Y_n$ defined in (\ref{eq:Yt}) and (\ref{eq:Zt})
where (\ref{eq:mut_Wt}) is replaced by
\begin{equation}\label{eq:mut_simple}
\mu_t^\star=\exp(W_t^\star) \textrm{ with } W_t^\star=\beta_0^\star+Z_t^\star.
\end{equation}
\textcolor{black}{Note that some theoretical results have already been obtained in this framework (no covariates and $q=1$)
by \cite{davis:dunsmuir:streett:2003} and \cite{davis:dunsmuir:street:2005}. However, here, we provide, on the one hand, a more detailed version
of the proof of these results and on the other hand, a proof of the consistency of $\gamma_1^\star$ based on a stochastic equicontinuity result.
 We limit ourselves to this framework since the more general one is a framework
where the consistency is much more tricky to handle and is beyond the scope of this paper.}

\begin{theo}\label{theo:MA1}
Assume that $Y_1,\dots,Y_n$ satisfy the model defined by (\ref{eq:Yt}), (\ref{eq:mut_simple}) and (\ref{eq:Zt}) with $q=1$ and $\gamma_1^\star\in\Gamma$ where $\Gamma$ is a compact set
of $\mathbb{R}$ which does not contain 0. Assume also that $(W_t^\star)$ 
is started with its stationary invariant distribution. Let $\widehat{\gamma}_1$ be defined by:
$$
\widehat{\gamma}_1=\textrm{Argmax}_{\gamma_1\in\Gamma}\; L(\beta_0^\star,\gamma_1),
$$
where 
\begin{equation}\label{eq:L:beta_0}
L(\beta_0^\star,\gamma_1)=\sum_{t=1}^n\left(Y_t W_t(\beta_0^\star,\gamma_1)-\exp(W_t(\beta_0^\star,\gamma_1)\right),
\end{equation}
with 
\begin{equation}\label{eq:W_Z}
W_t(\beta_0^\star,\gamma_1)=\beta_0^\star+Z_t(\gamma_1)=\beta_0^\star+\gamma_1 E_{t-1}(\gamma_1), 
%\textrm{ and } E_{t}(\gamma_1)=Y_t\exp(-W_{t})-1,\; t>0
\end{equation}

$$
E_{t-1}(\gamma_1)=Y_{t-1}\exp(-W_{t-1}(\beta_0^\star,\gamma_1))-1, \textrm{ if } t>1 \textrm{ and } E_{t-1}(\gamma_1)=0, \textrm{ if } t\leq 1.
$$

Then $\widehat{\gamma}_1\stackrel{p}{\longrightarrow}\gamma_1^\star$, as $n$ tends to infinity, where $\stackrel{p}{\longrightarrow}$ denotes the convergence in probability.
\end{theo}

The proof of Theorem \ref{theo:MA1} is based on the following propositions which are proved in Section \ref{sec:proofs}. These propositions are the classical arguments for establishing consistency
results of maximum likelihood estimators. Note that we shall explain in the proof of Proposition \ref{prop1}
why a stationary invariant distribution for $(W_t^\star)$ does exist. The main tools used for proving Propositions \ref{prop1} and \ref{prop3} are the Markov property and the ergodicity of $(W_t^\star)$.

\begin{prop}\label{prop1}
For all fixed $\gamma_1$, under the assumptions of Theorem \ref{theo:MA1}, 
\begin{equation}\label{eq:conv}
\frac1n L(\beta_0^\star,\gamma_1)\stackrel{p}{\longrightarrow} 
\mathcal{L}(\gamma_1):=\mathbb{E}\left[Y_3 W_3(\beta_0^\star,\gamma_1)-\exp(W_3(\beta_0^\star,\gamma_1)\right], \textrm{ as $n$ tends to infinity.}
\end{equation}
\end{prop}

\begin{prop}\label{prop2}
The function $\mathcal{L}$ defined in (\ref{eq:conv}) has a unique maximum at the true parameter $\gamma_1=\gamma_1^\star$.
\end{prop}

\begin{prop}\label{prop3}
Under the assumptions of Theorem \ref{theo:MA1}
$$\sup_{\gamma_1\in\Gamma}\left|\frac{L(\beta_0^\star,\gamma_1)}{n}-\mathcal{L}(\gamma_1)\right|\stackrel{p}{\longrightarrow}0, \textrm{ as $n$ tends to infinity,}$$
where $\mathcal{L}(\gamma_1)$ is defined in (\ref{eq:conv}).
\end{prop}

% \begin{theo}\label{theo:MAq}
% Assume that $Y_1,\dots,Y_n$ satisfy (\ref{eq:Yt}), (\ref{eq:mut_simple}) and (\ref{eq:Zt}) and that there exists $k_0$ such that $\gamma_{k_0}^\star\neq 0$ \textcolor{black}{A VERIFIER}. 
% Let $\widehat{\boldsymbol{\gamma}}=(\widehat{\gamma}_1,\dots,\widehat{\gamma}_q)$ be defined by:
% $$
% \widehat{\boldsymbol{\gamma}}=\textrm{Argmax}_{\boldsymbol{\gamma}} L(\beta_0^\star,\boldsymbol{\gamma}),
% $$
% where 
% $$
% L(\beta_0^\star,\boldsymbol{\gamma})=\sum_{t=1}^n\left(Y_t W_t(\beta_0^\star,\boldsymbol{\gamma})-\exp(W_t(\beta_0^\star,\boldsymbol{\gamma})\right).
% $$
% Then $\widehat{\boldsymbol{\gamma}}\stackrel{p}{\longrightarrow}(\gamma_1^\star,\dots,\gamma_q^\star)$ as $n$ tends to infinity, where $\stackrel{p}{\longrightarrow}$ denotes the convergence in probability.
% \end{theo}

% estimation de beta à inclure ????
%   dire qu'on ne peut pas mettre les $x_{i,t}$ (????) à cause de la stationarité

\section{Numerical experiments}\label{sec:num}

\subsection{Statistical performance}

The goal of this section is to investigate the performance of our method both from a statistical and a numerical point of view.

\subsubsection{Estimation of the ARMA part coefficients when $p=0$}

In this section, we investigate the statistical performance of our methodology in the case where $Y_1,\dots,Y_n$ satisfy the model
defined by (\ref{eq:Yt}), (\ref{eq:mut_Wt}) and (\ref{eq:Zt}) for $n$ in $\{50,100,250,500,1000\}$
in the case where $p=0$, namely when there are no covariates and for $q$ in
$\{1,2,3\}$. The results are displayed in Figures \ref{fig:estim_beta}, \ref{fig:estim:gam1} and \ref{fig:estim:gam2_3}. We can see from these figures that
the accuracy of the parameters estimations is improved when $n$ increases.

\begin{figure}[!htbp]
  \includegraphics[scale=0.28]{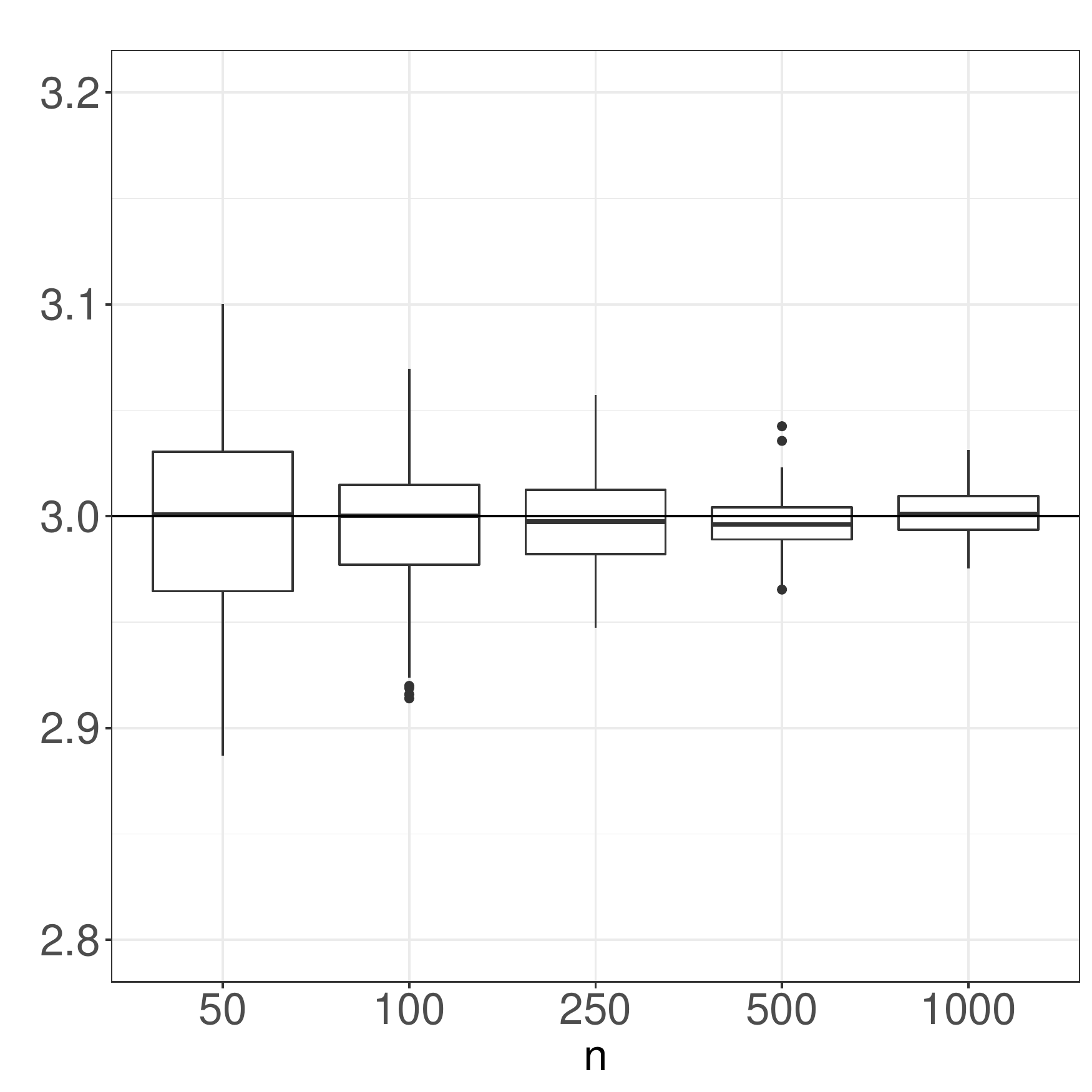}
  \includegraphics[scale=0.28]{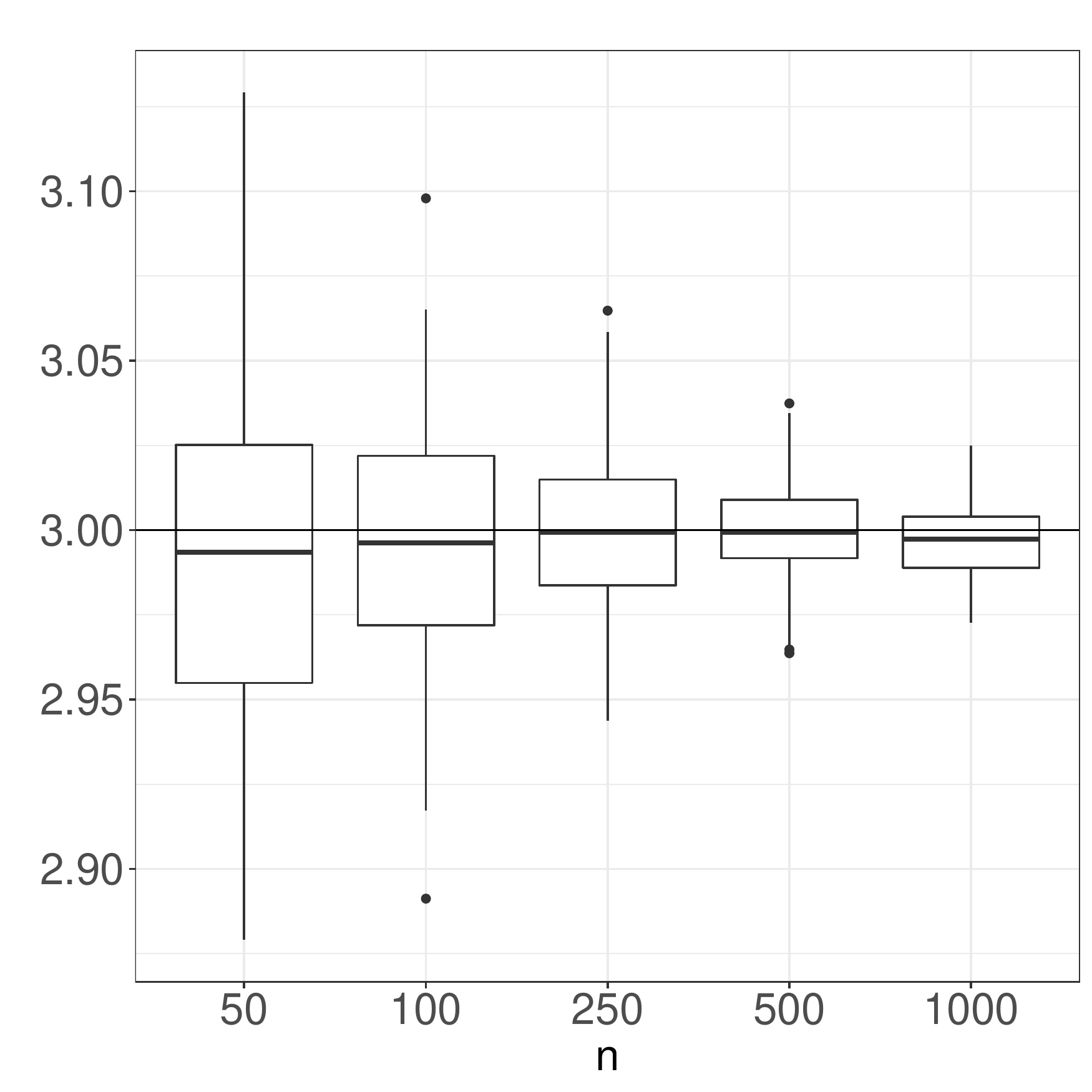}
  \includegraphics[scale=0.28]{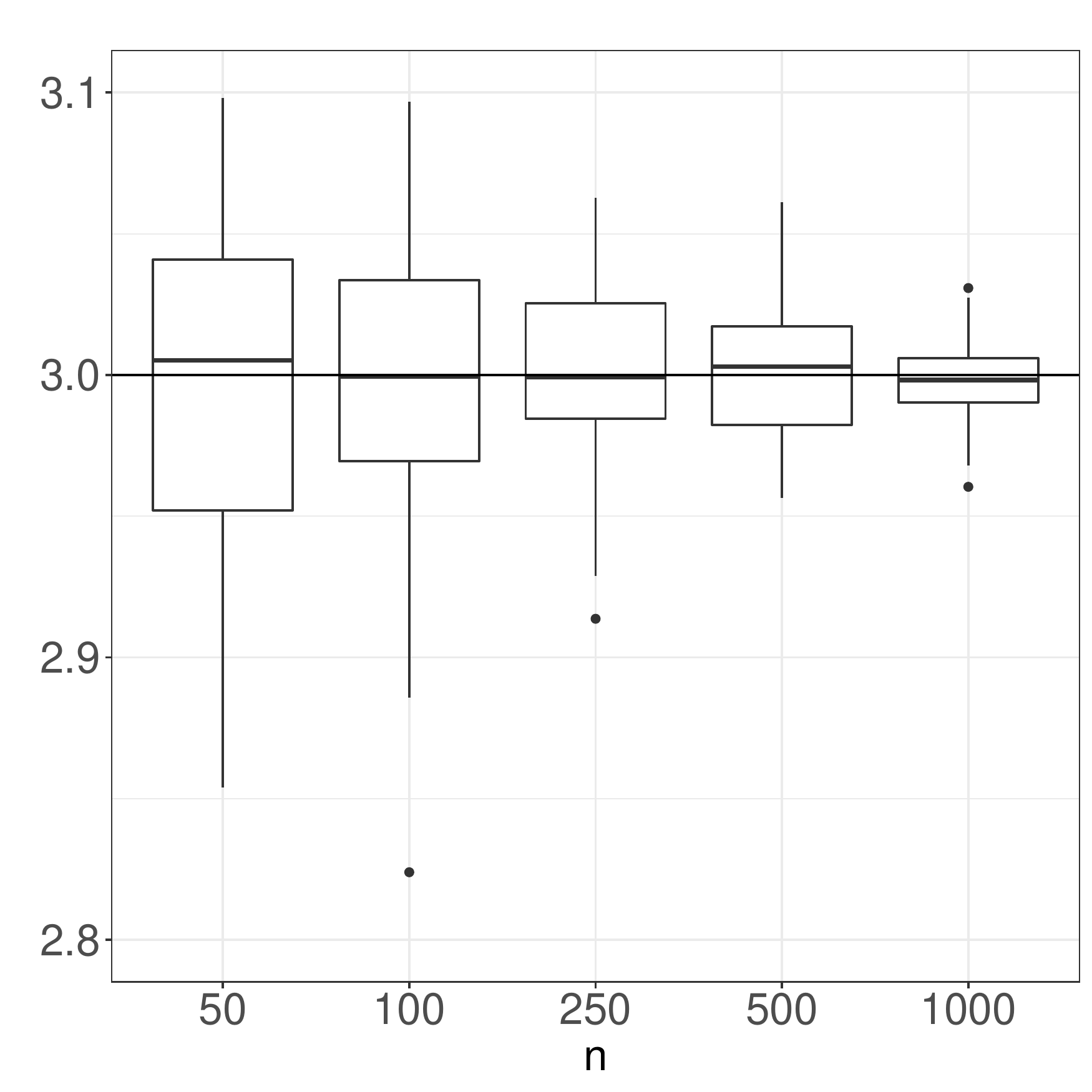}
  \caption{Boxplots for the estimations of $\beta_0^\star=3$ in Model (\ref{eq:mut_Wt}) with no regressor and $q=1$ (left), $q=2$ (middle) and $q=3$ (right). 
The horizontal lines correspond to the value of $\beta_0^\star$.\label{fig:estim_beta}}
\end{figure}

\begin{figure}[!htbp]
  \includegraphics[scale=0.28]{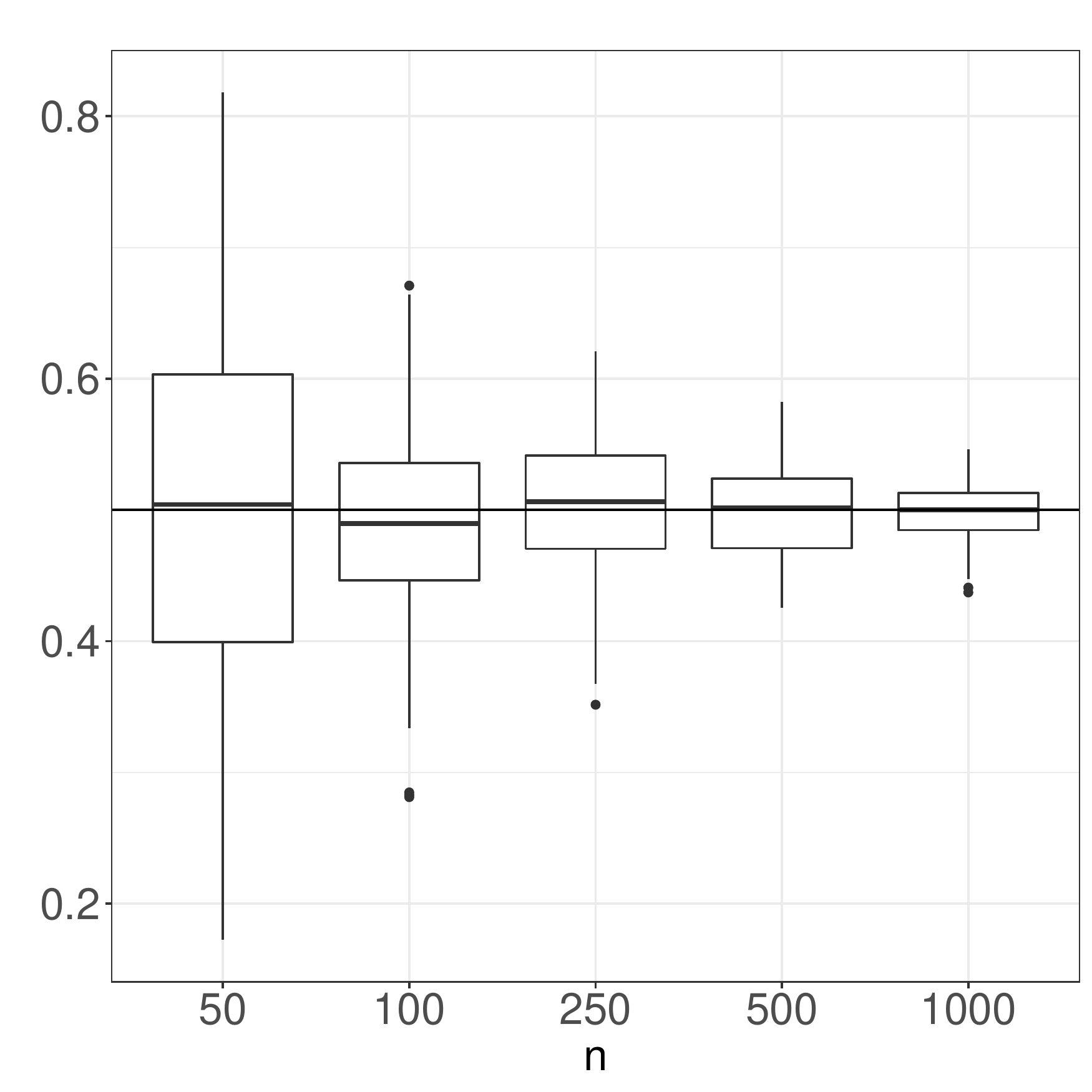}
  \includegraphics[scale=0.28]{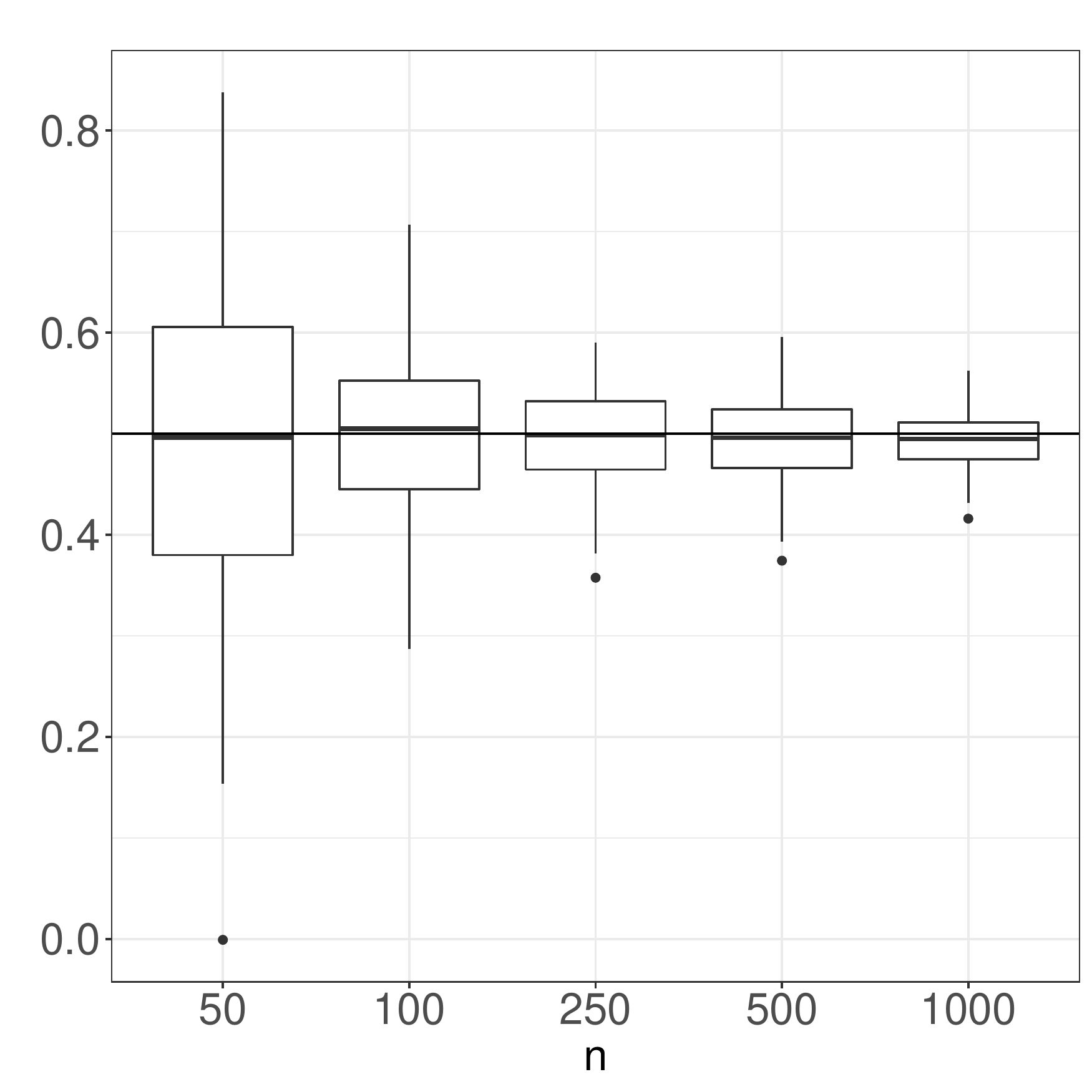}
  \includegraphics[scale=0.28]{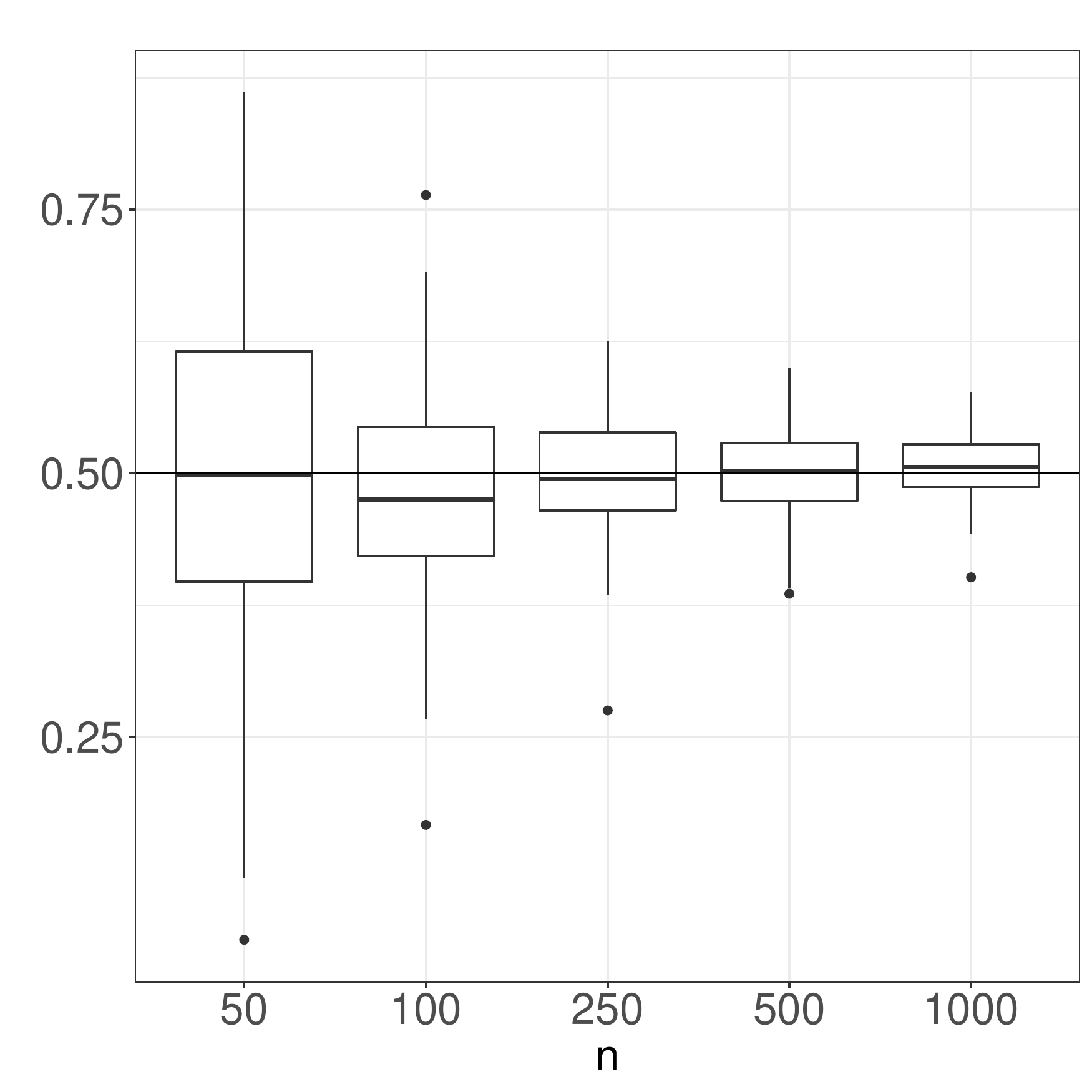}
  \caption{Boxplots for the estimations of $\gamma_1^\star=0.5$ in Model (\ref{eq:mut_Wt}) with no regressor and $q=1$ (left), $q=2$ (middle) and $q=3$ (right).
  \textcolor{black}{The horizontal lines correspond to the value of $\gamma_1^\star$.} \label{fig:estim:gam1}}
\end{figure}

  \begin{figure}[!htbp]
  \includegraphics[scale=0.28]{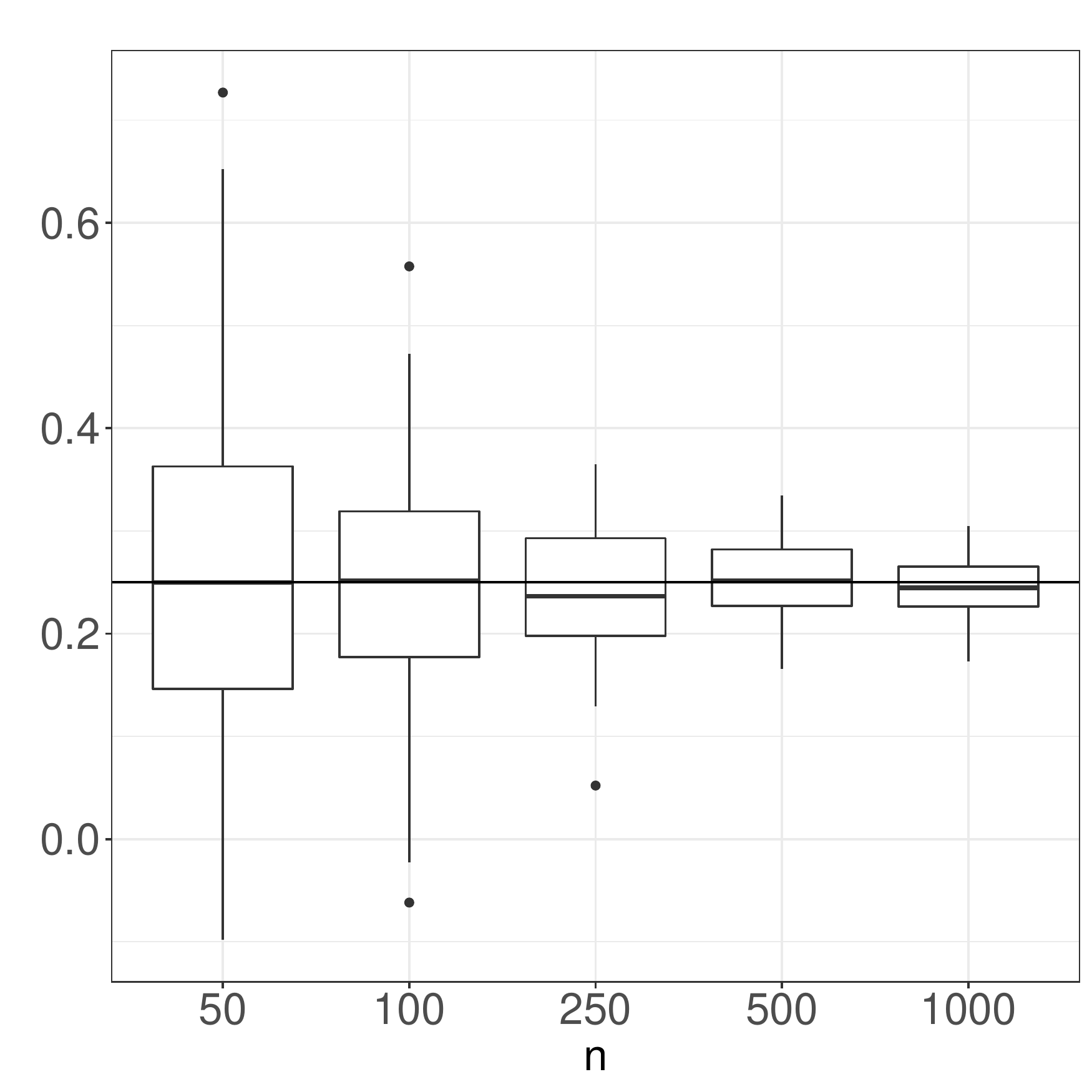}
  \includegraphics[scale=0.28]{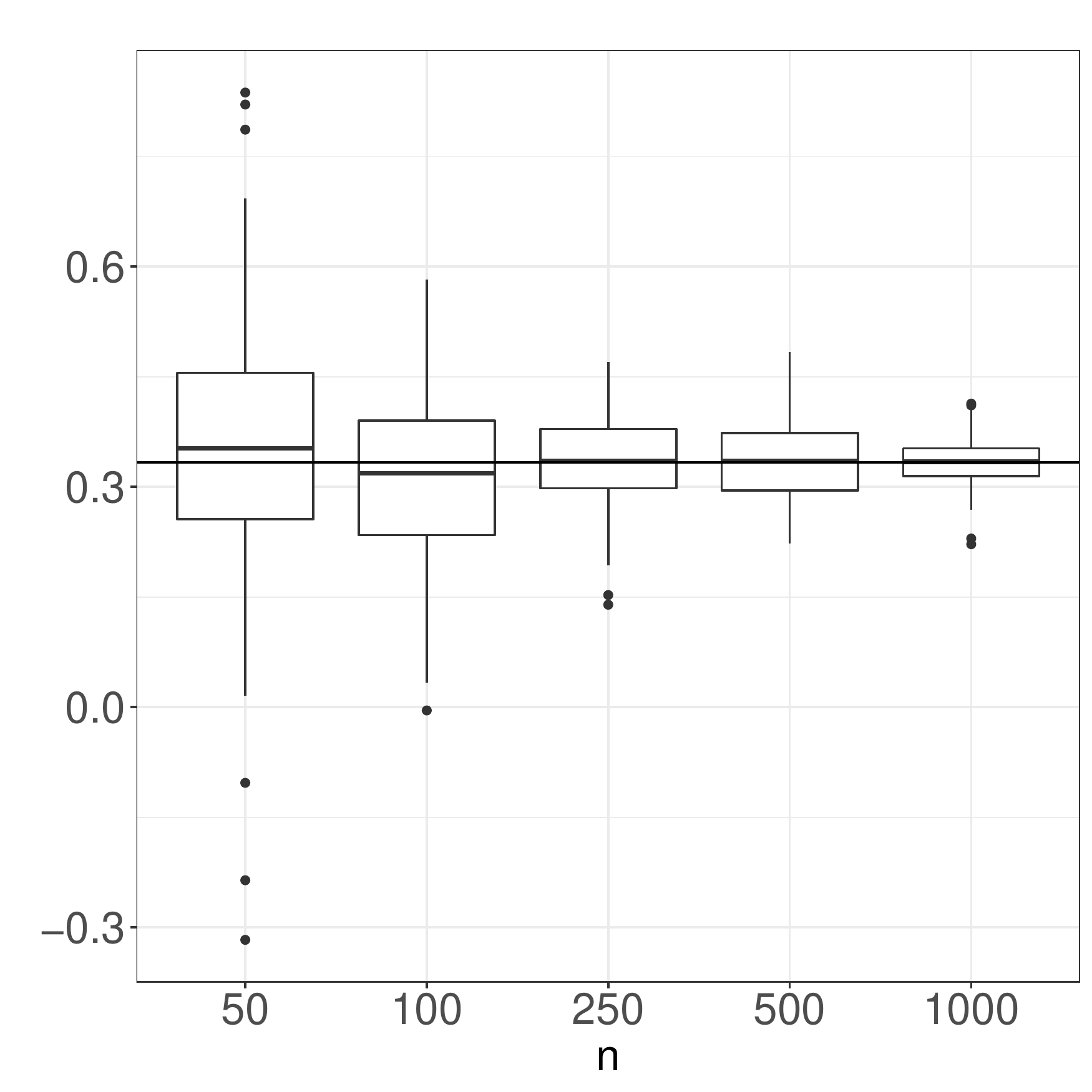}
  \includegraphics[scale=0.28]{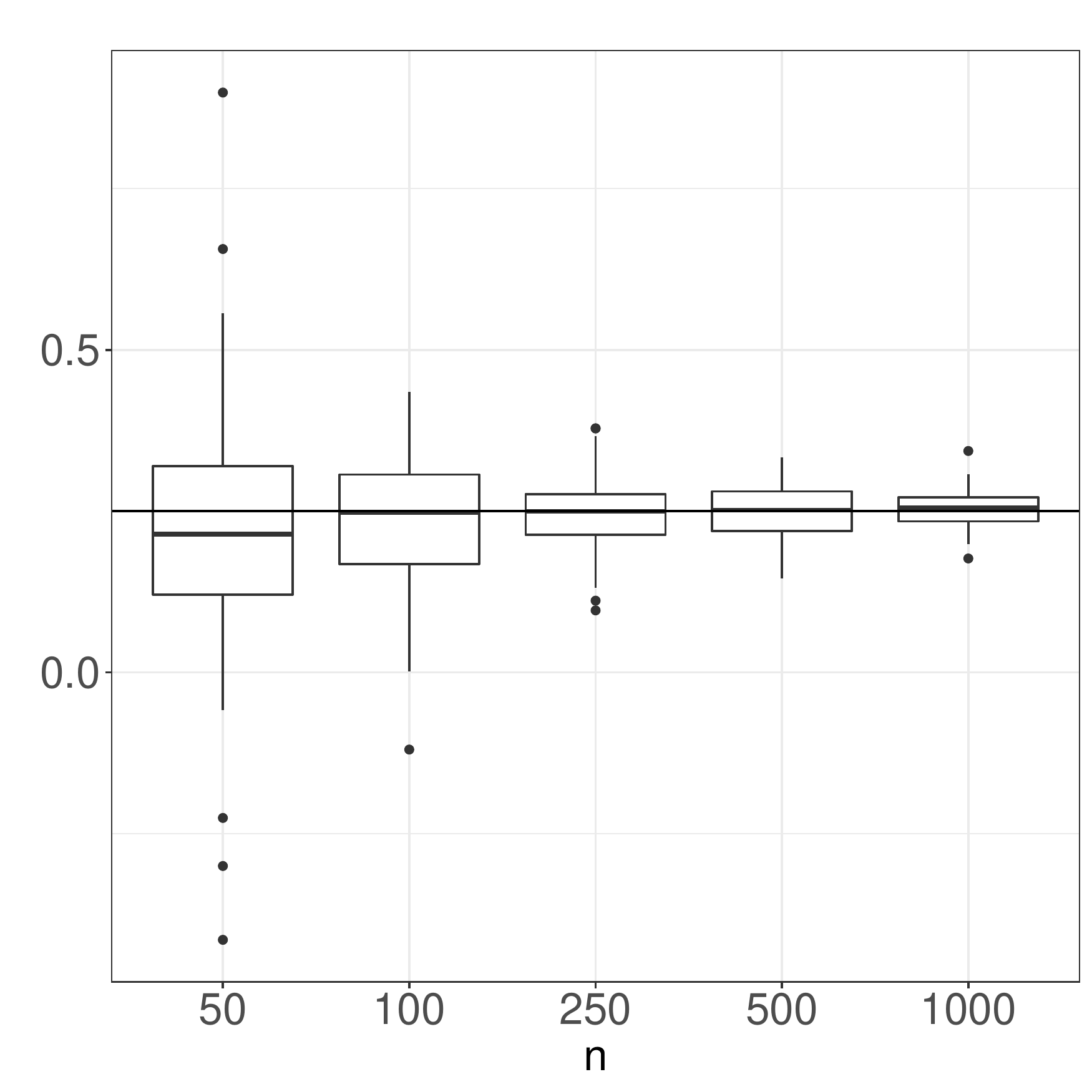}
  \caption{Boxplots for the estimations of $\gamma_2^\star=1/4$ in Model (\ref{eq:mut_Wt}) with no regressor and $q=2$ (left), $\gamma_2^\star=1/3$ in Model (\ref{eq:mut_Wt}) 
with no regressor $q=3$ (middle) and of $\gamma_3^\star=1/4$
    in Model (\ref{eq:mut_Wt}) with no regressor and $q=3$ (right).
\textcolor{black}{The horizontal lines correspond to the true values of the parameters.}
  \label{fig:estim:gam2_3}}
\end{figure}

Moreover, it has to be noticed that in this particular context where there are no covariates ($p=0$), the
performance of our approach in terms of parameters estimation is similar to the one of the package \texttt{glarma}
described in \cite{glarma:package}.

%% MA(3) ou MA(4) : boxplots qui montrent que l'estimation s'améliore quand $n$ augmente sans régresseurs

%% comparaison au package GLARMA, dire qu'on est capable de calculer des AIC aussi et que memes résultats que nous.
%%  Prendre un cas de MA non inversible que GLARMA ne sait pas gérer

\subsubsection{Sparse estimation of the ${\beta}_i^\star$}\label{sec:sparse_estim}

In this section, we assess the performance of our methodology in the case where $Y_1,\dots,Y_n$ satisfy the model
defined by (\ref{eq:Yt}), (\ref{eq:mut_Wt}) and (\ref{eq:Zt}) for $n=1000$,
$q\in\{1,2,3\}$ and $p=100$. We shall moreover assume that the sparsity in the
$\beta_i^\star$ is very high, namely all the $\beta_i^\star$ are assumed to be equal to zero except for five of them which are equal to
1.739, 0.387, 0.295, -0.644 and -0.135.
 The corresponding results are displayed in Figures \ref{fig:roc1}, \ref{fig:roc2} and \ref{fig:roc3}.

 \begin{figure}[!htbp]
   \includegraphics[scale=0.3]{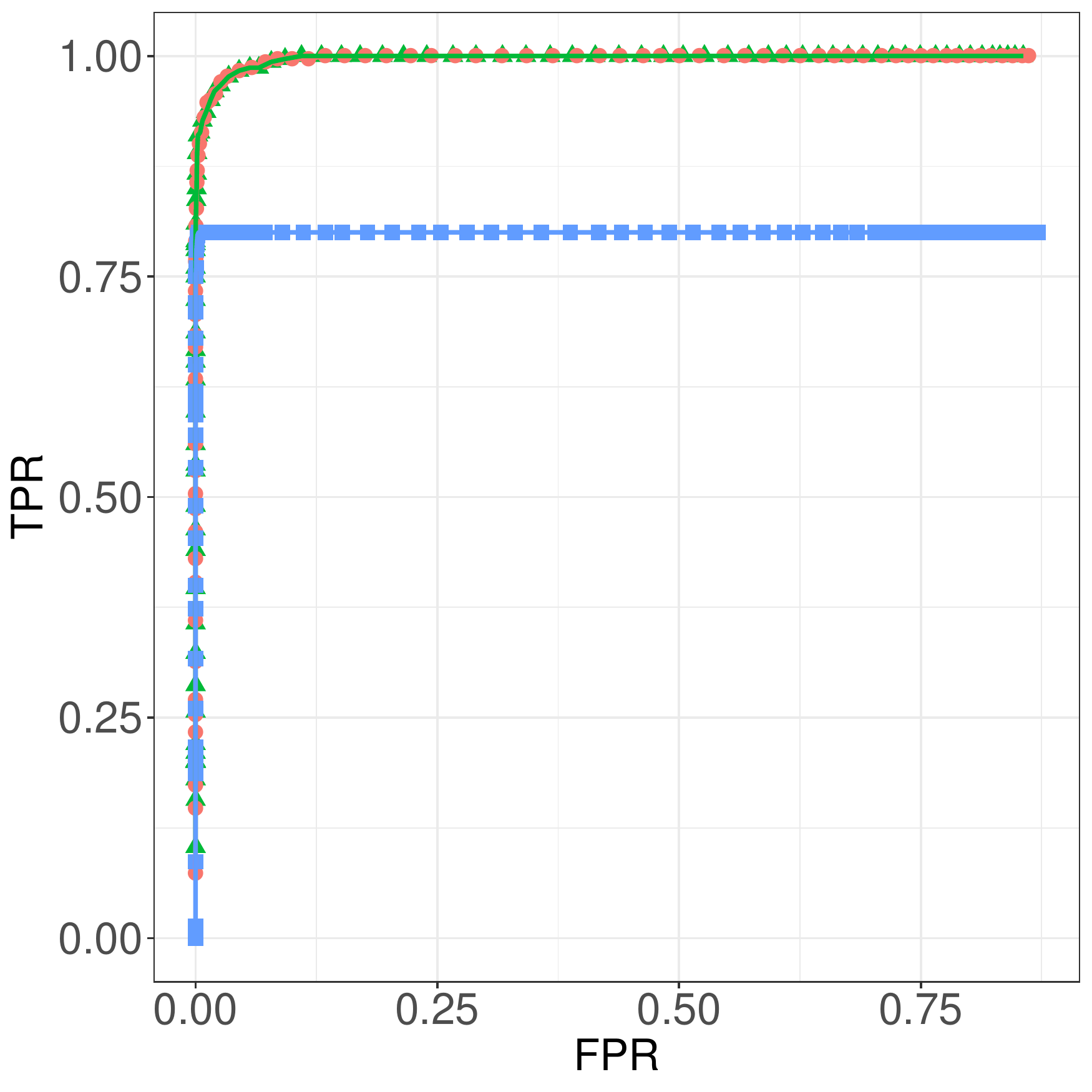}
\includegraphics[scale=0.32]{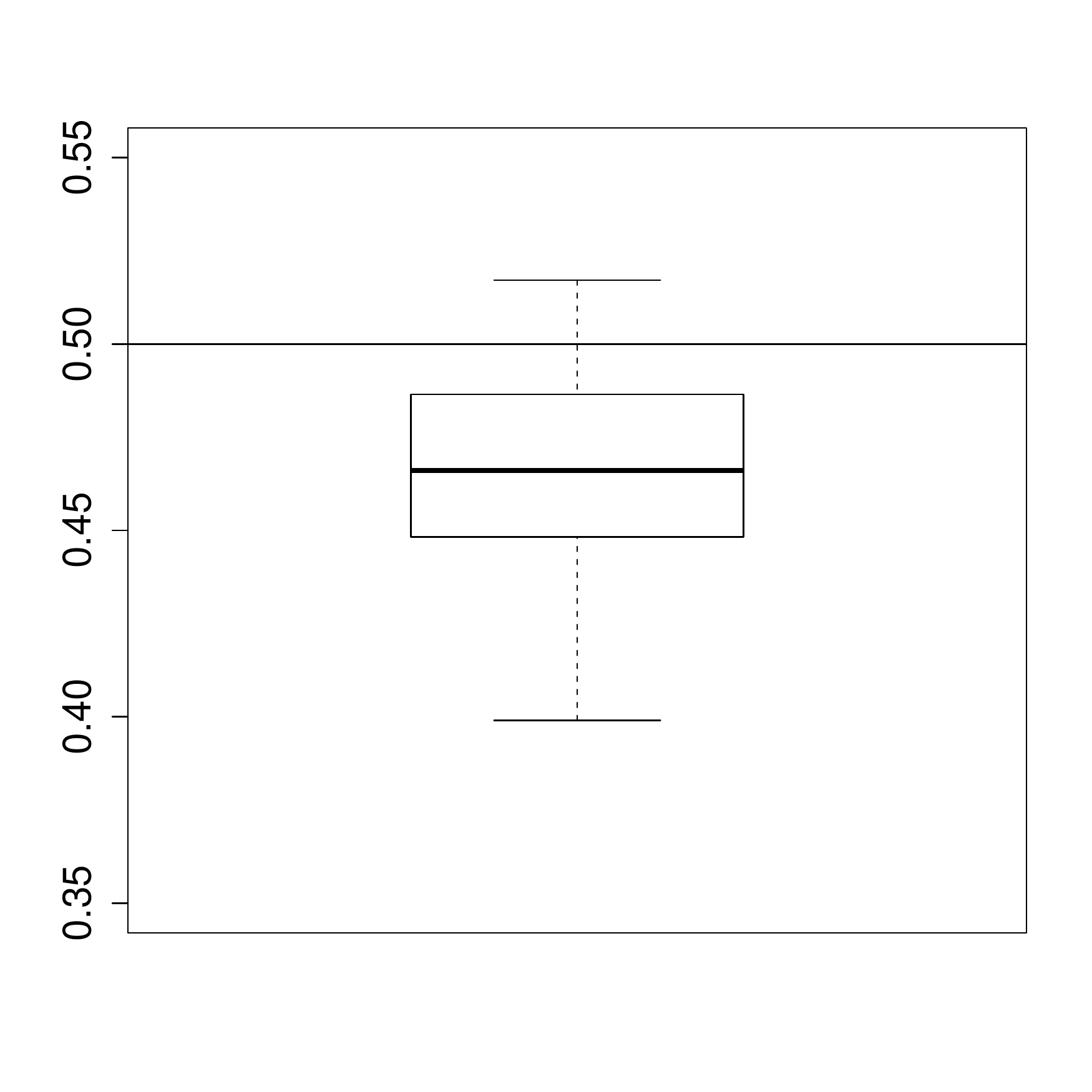}
\caption{ROC curves for recovering the support of $\boldsymbol{\beta}^\star$  in Model (\ref{eq:mut_Wt}) with $q=1$ (left) and boxplot for the estimation of
  ${\gamma_1}^\star$ in the same model (right).
 The ROC curve when ${\gamma_1}^\star$ is known (resp. unknown) is in red (resp. green) and in the model where ${\gamma_1}^\star=0$ in blue.
\textcolor{black}{The horizontal lines correspond to the true values of the parameters.}\label{fig:roc1}}
 \end{figure}
  
 \begin{figure}[!htbp]
   \includegraphics[scale=0.28]{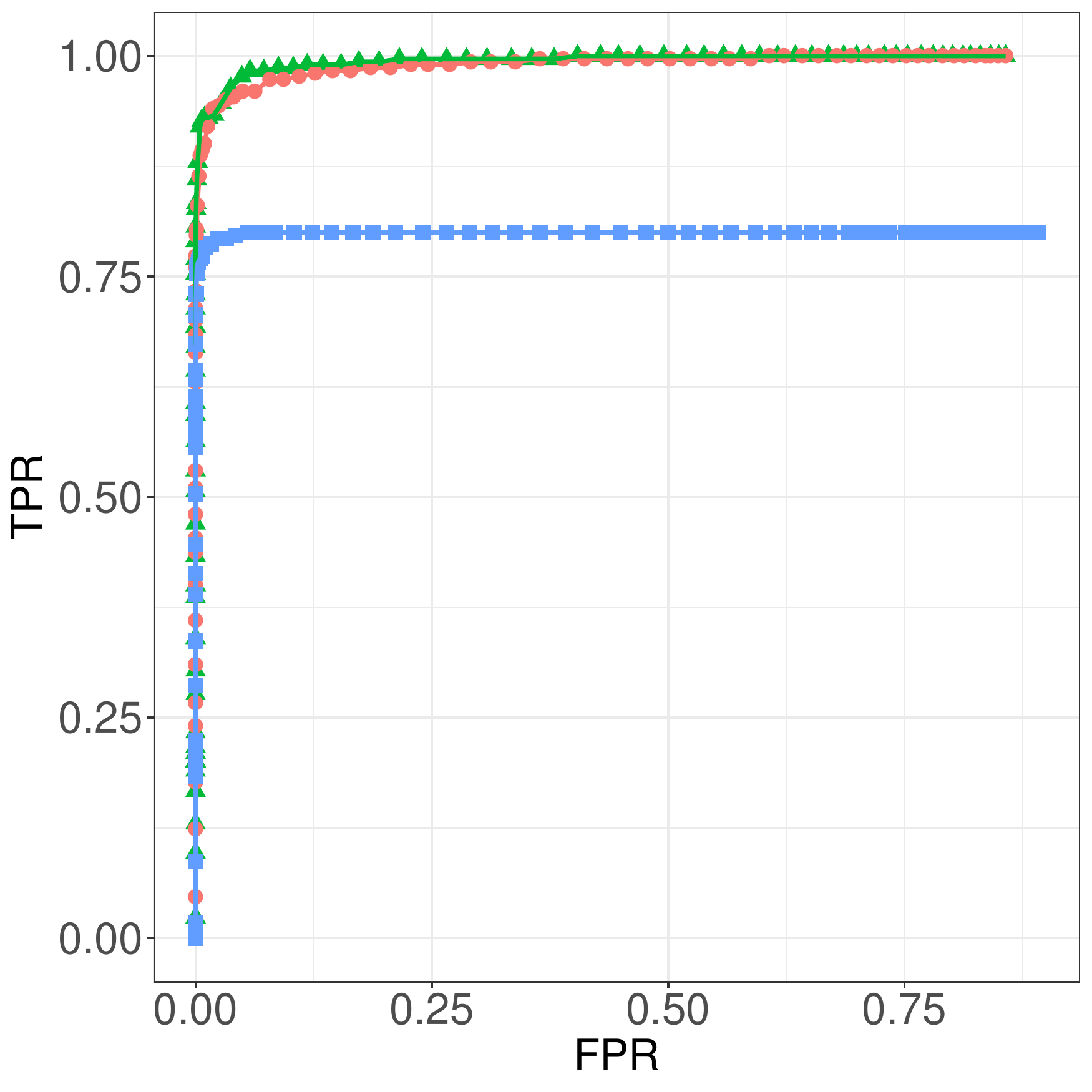}
\includegraphics[scale=0.28]{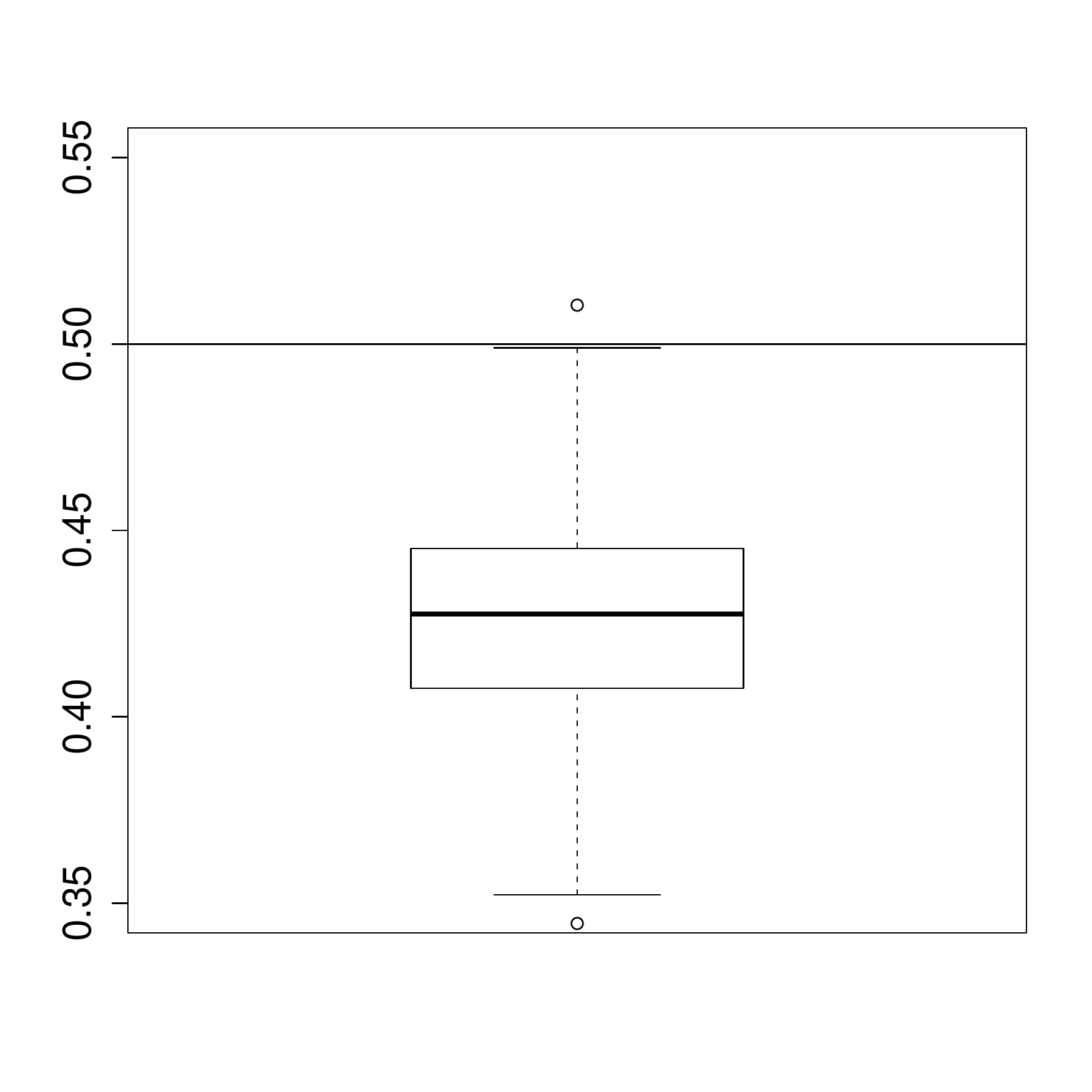}
\includegraphics[scale=0.28]{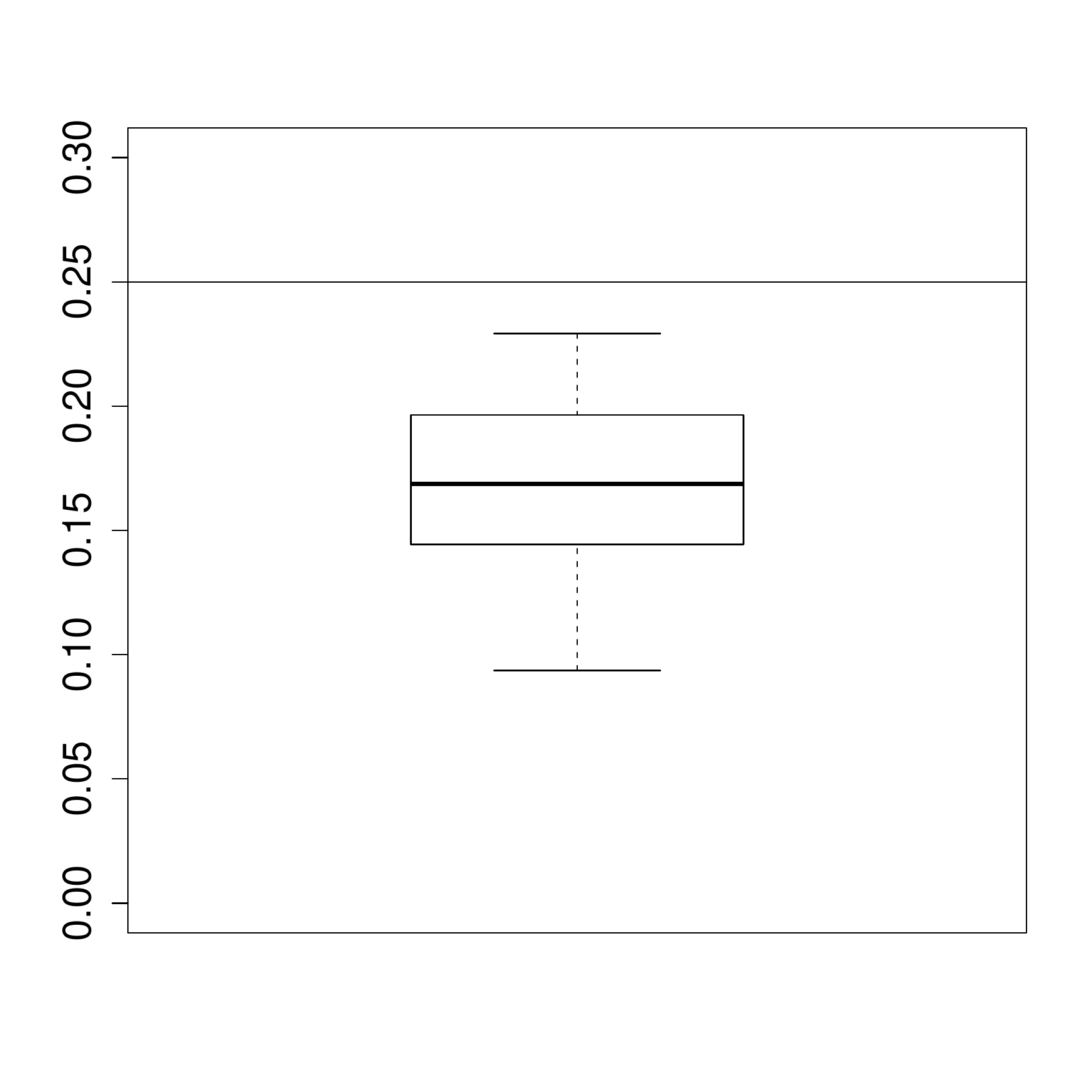}
\caption{ROC curves for recovering the support of $\boldsymbol{\beta}^\star$  in Model (\ref{eq:mut_Wt}) with $q=2$ (left), boxplots for the
  estimation of ${\gamma_1}^\star$ (middle)  and ${\gamma_2}^\star$ (right)  in the same model.
The ROC curve when ${\gamma_1}^\star$ and ${\gamma_2}^\star$ are known (resp. unknown) is in red (resp. green) and in the model where
${\gamma_1}^\star={\gamma_2}^\star=0$ in blue.
\textcolor{black}{The horizontal lines correspond to the true values of the parameters.}\label{fig:roc2}}
 \end{figure}

\begin{figure}[!htbp]
   \includegraphics[scale=0.3]{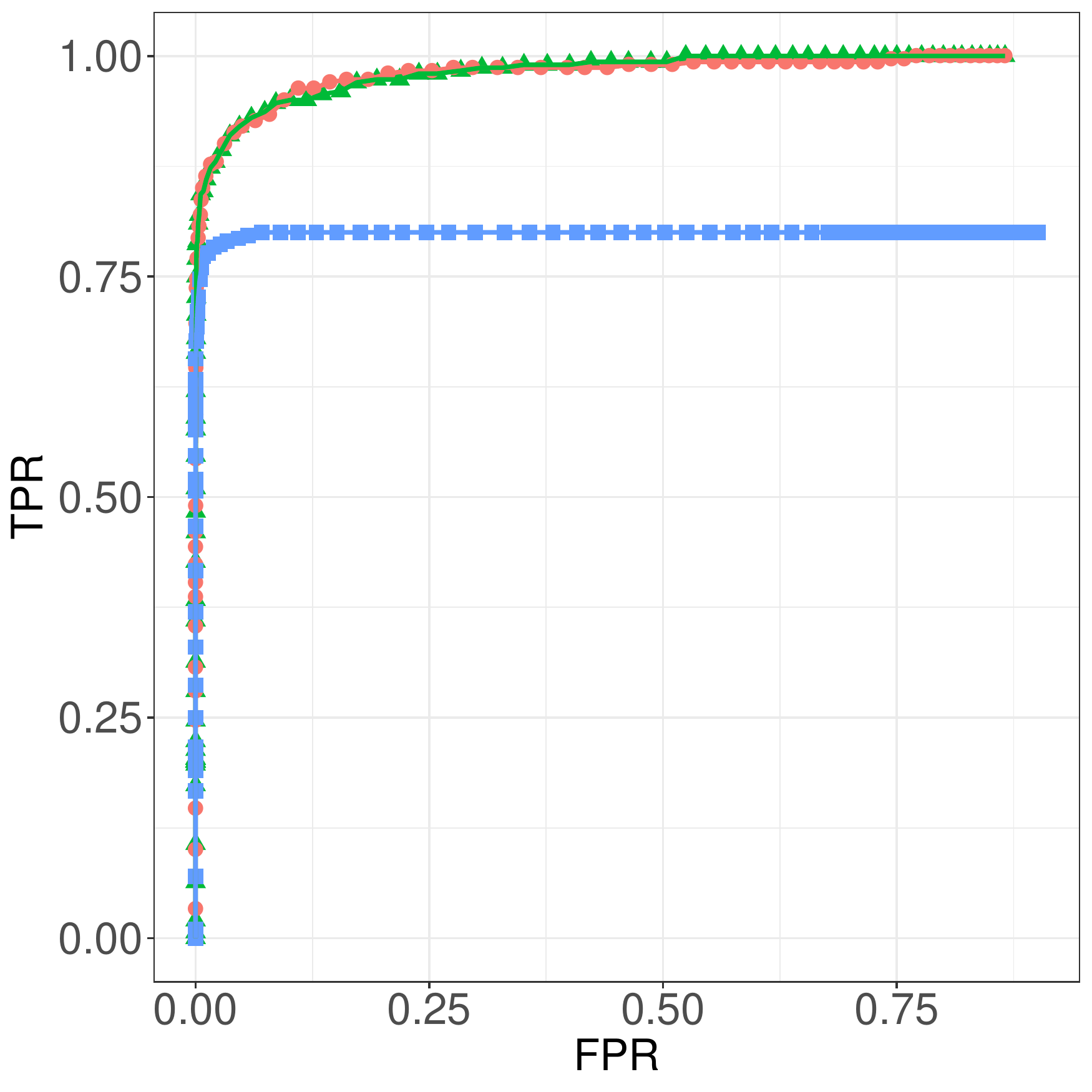}
\includegraphics[scale=0.3]{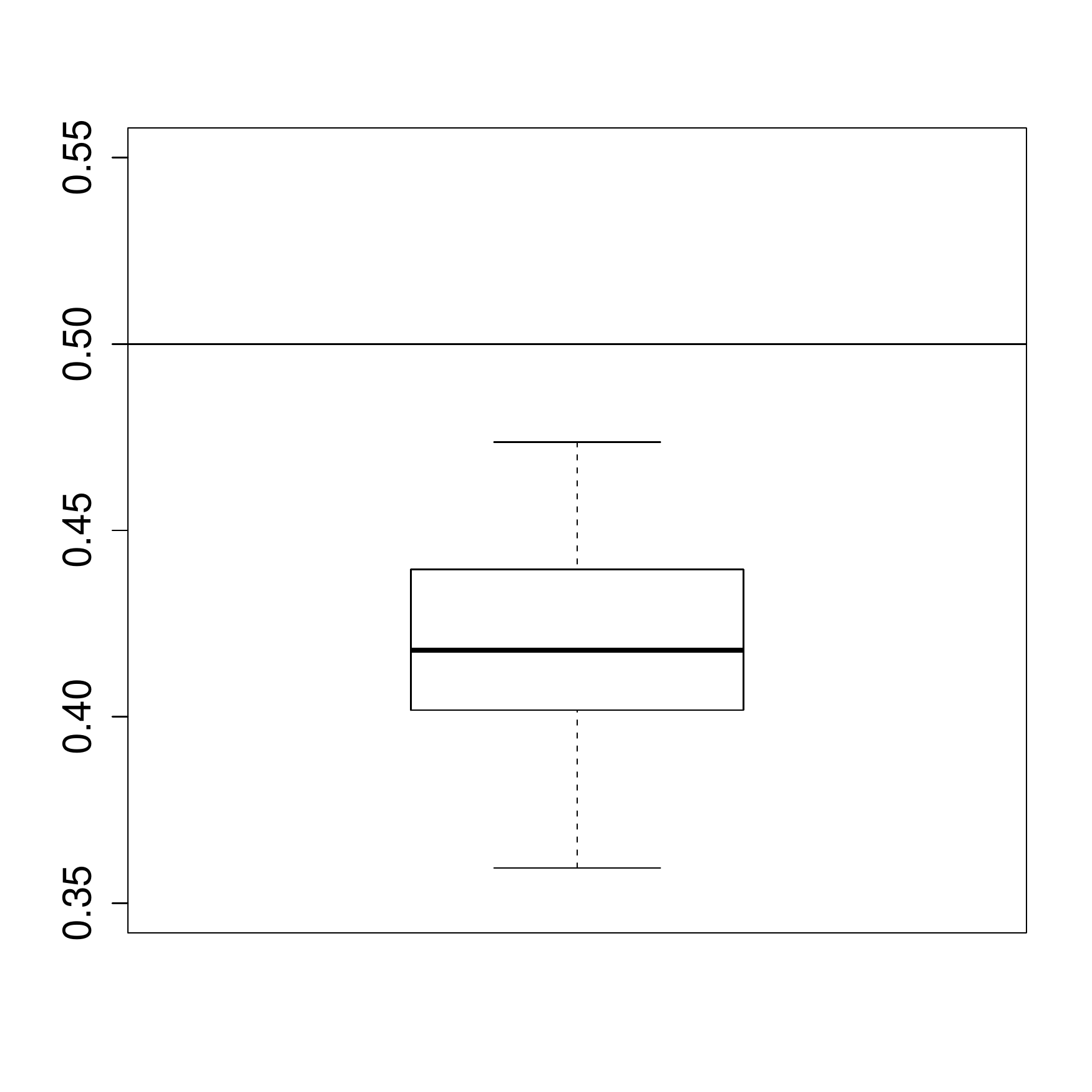}\\
\includegraphics[scale=0.3]{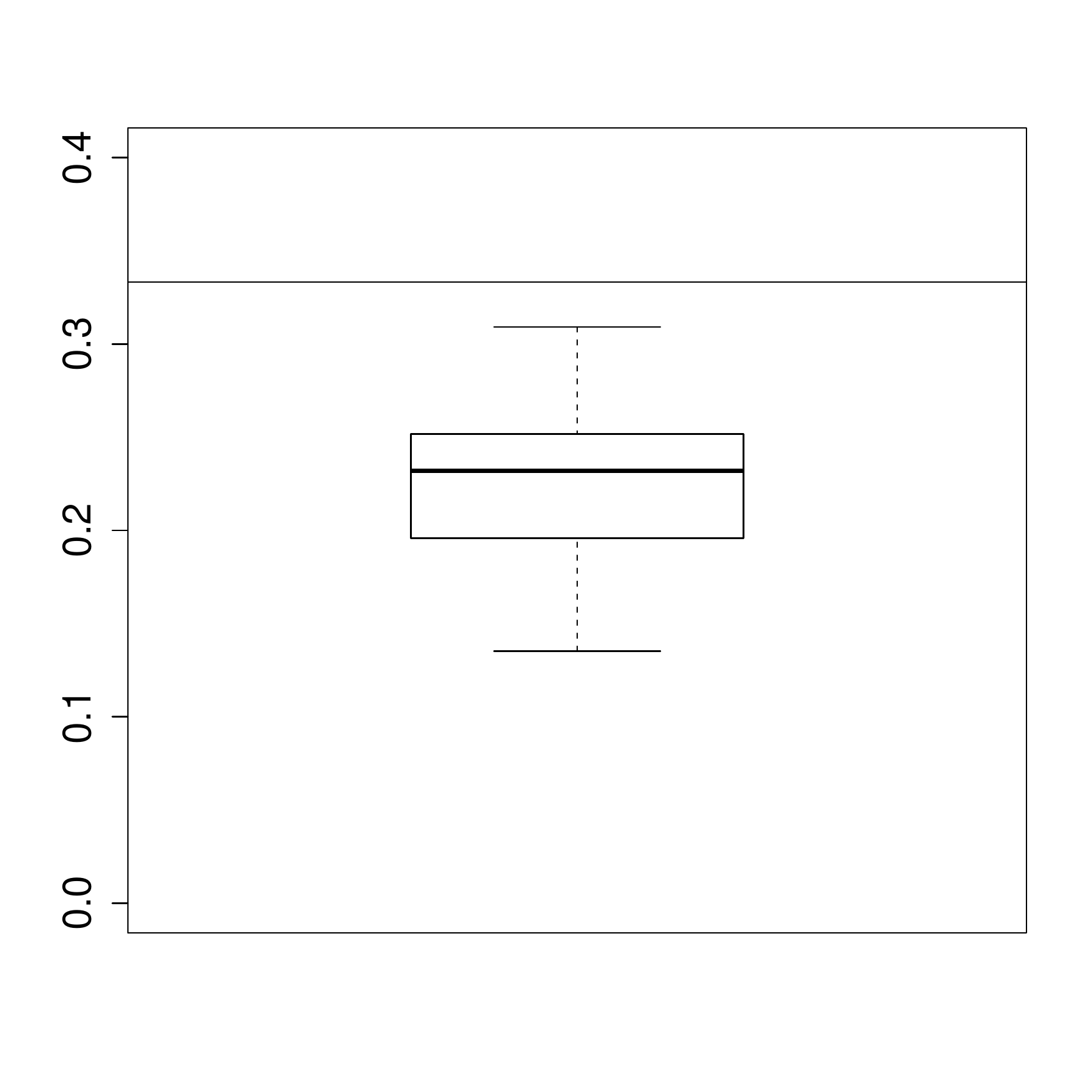}
\includegraphics[scale=0.3]{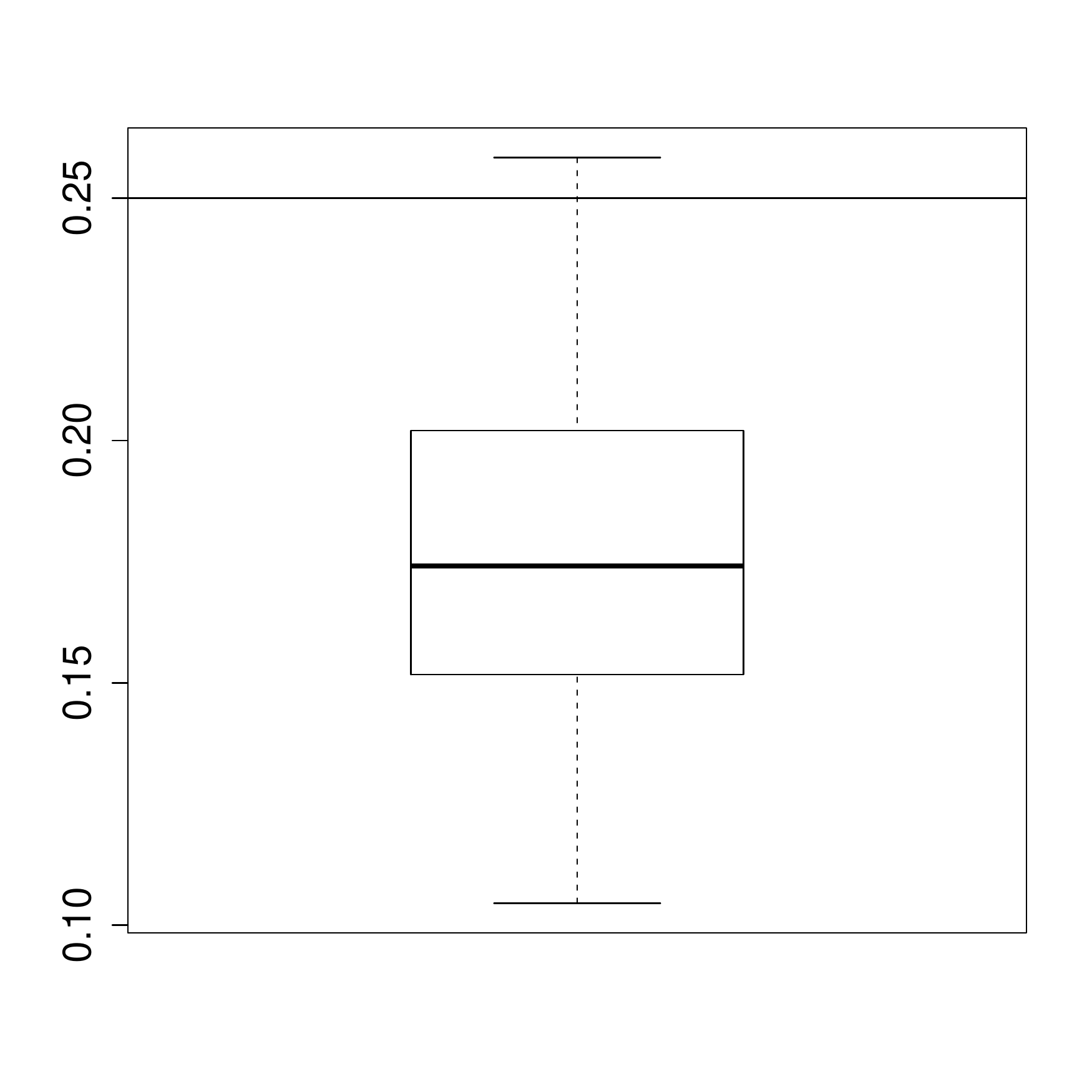}
   \caption{ROC curves for recovering the support of $\boldsymbol{\beta}^\star$  in Model (\ref{eq:mut_Wt}) with $q=3$ (top left), boxplots for the estimation of ${\gamma_1}^\star$ (top right), ${\gamma_2}^\star$ (bottom left) and ${\gamma_3}^\star$ (bottom right) in the same model.
 The ROC curve when ${\gamma_1}^\star$, ${\gamma_2}^\star$ and ${\gamma_3}^\star$ are known (resp. unknown) is in red (resp. green) and in the model where 
 ${\gamma_1}^\star={\gamma_2}^\star={\gamma_3}^\star=0$ in blue. \textcolor{black}{The horizontal lines correspond to the true values of the parameters.}
   \label{fig:roc3}}
 \end{figure}

The ROC curves of Figures \ref{fig:roc1}, \ref{fig:roc2} and \ref{fig:roc3}
display the True Positive Rate (TPR) with respect to the False Positive Rate (FPR). On the one hand, we can see from these figures that
the performance of our methodology
when $\boldsymbol{\gamma}^\star$ is known is on a par
with the one of our methodology when $\boldsymbol{\gamma}^\star$ is  unknown.
On the other hand,  our methodology outperforms the variable selection approach
described in \cite{friedman:hastie:tibshirani:2010} which assumes that the observations are the realizations of a Poisson 
distribution but does not take into account the dependence between the observations.

We can also observe from these figures that the performance of our methodology is not altered by the underestimation
of $\boldsymbol{\gamma}^\star$ in the different situations: $q=1$, 2 or 3.

%% comparaison avec GLARMA (détérioration des gamma si beaucoup de zéros dans les beta, on espère), comparaison avec glmnet (discret) pour la sélection de variables avec des courbes ROC, on
%% espere que le fait que cette methode ne tienne pas compte des gamma va deteriorer les resultats.

\subsubsection{Choice of $\lambda$}\label{sec:lambda}

In order to improve our methodology, we propose hereafter a strategy for tuning the parameter $\lambda$ appearing in (\ref{eq:beta_hat}). 

We first take the smallest $\lambda$ provided by the \texttt{glmnet} package for computing (\ref{eq:beta_hat}). This $\lambda$ denoted $\lambda_{\textrm{min}}$
is then used in the stability selection procedure proposed by \cite{meinshausen:buhlmann:2010}
which guarantees the robustness of the selected variables. This latter approach
can be described as follows.
The vector $\mathcal{Y}$ defined in (\ref{eq:def_Y_X}) is randomly split into several subsamples of size $(p+1)/2$, which corresponds to the half of the length of $\mathcal{Y}$.
For each subsample, the LASSO criterion is applied with $\lambda=\lambda_{\textrm{min}}$ and the indices $i$ of the non null $\widehat{\beta}_i$ are stored. 
Then, for a given threshold, we keep in the final
 set of selected variables only the variables appearing a number of times larger than this threshold. 
 In practice, we generated $1000$ subsamples of $\mathcal{Y}$. 

Figure \ref{fig:freq_ronds} displays the results obtained when applying this strategy to observations $Y_1,\dots,Y_n$ satisfying the model
defined by (\ref{eq:Yt}), (\ref{eq:mut_Wt}) and (\ref{eq:Zt}) for $n=1000$,
$q=1$, $p=100$ and when only five coefficients $\beta_i^\star$ are not null. We can see from this figure that the positions of the non null coefficients are well retrieved for most of the thresholds
and that the number of false positive is higher when the threshold is too low. Based on this figure, taking a threshold equal to 0.9 seems to achieve
 an interesting trade-off between false and true positives.

\begin{figure}[!htbp]
\includegraphics[scale=0.4]{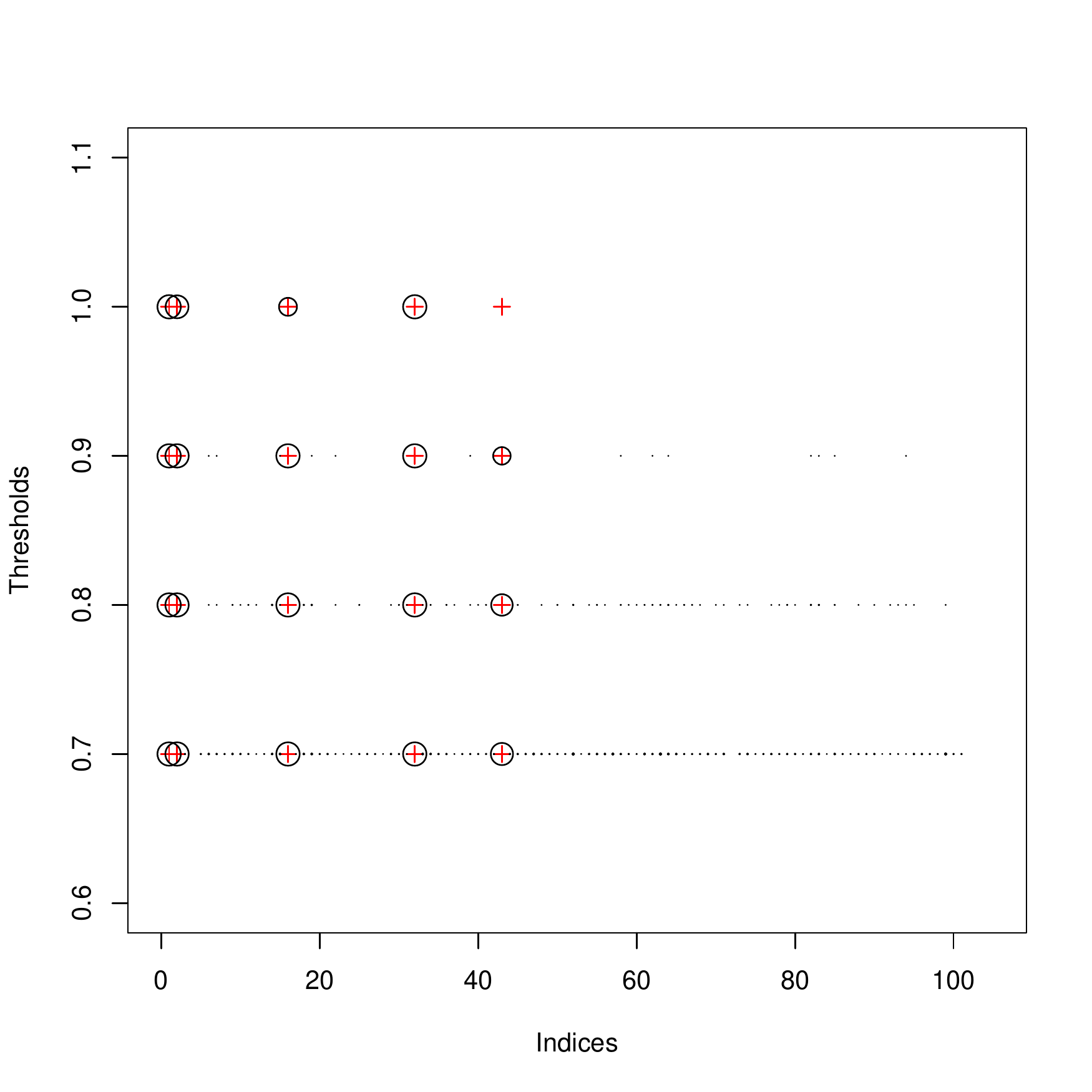}
\caption{Means of the selection frequencies of the indices of the final estimator of $\boldsymbol{\beta}^\star$ for different thresholds: 0.7, 0.8, 0.9 and 1
and based on 100 replications. 
The larger the size of circles the larger the frequency of considering the corresponding
coefficients as non null. The positions of the non null values of $\boldsymbol{\beta}^\star$ are displayed with red crosses.\label{fig:freq_ronds}}
\end{figure}

This choice is also confirmed by the results of Figure \ref{fig:freq_barres} which gives the means of selection frequencies for each position.

\begin{figure}[!htbp]
\includegraphics[scale=0.35]{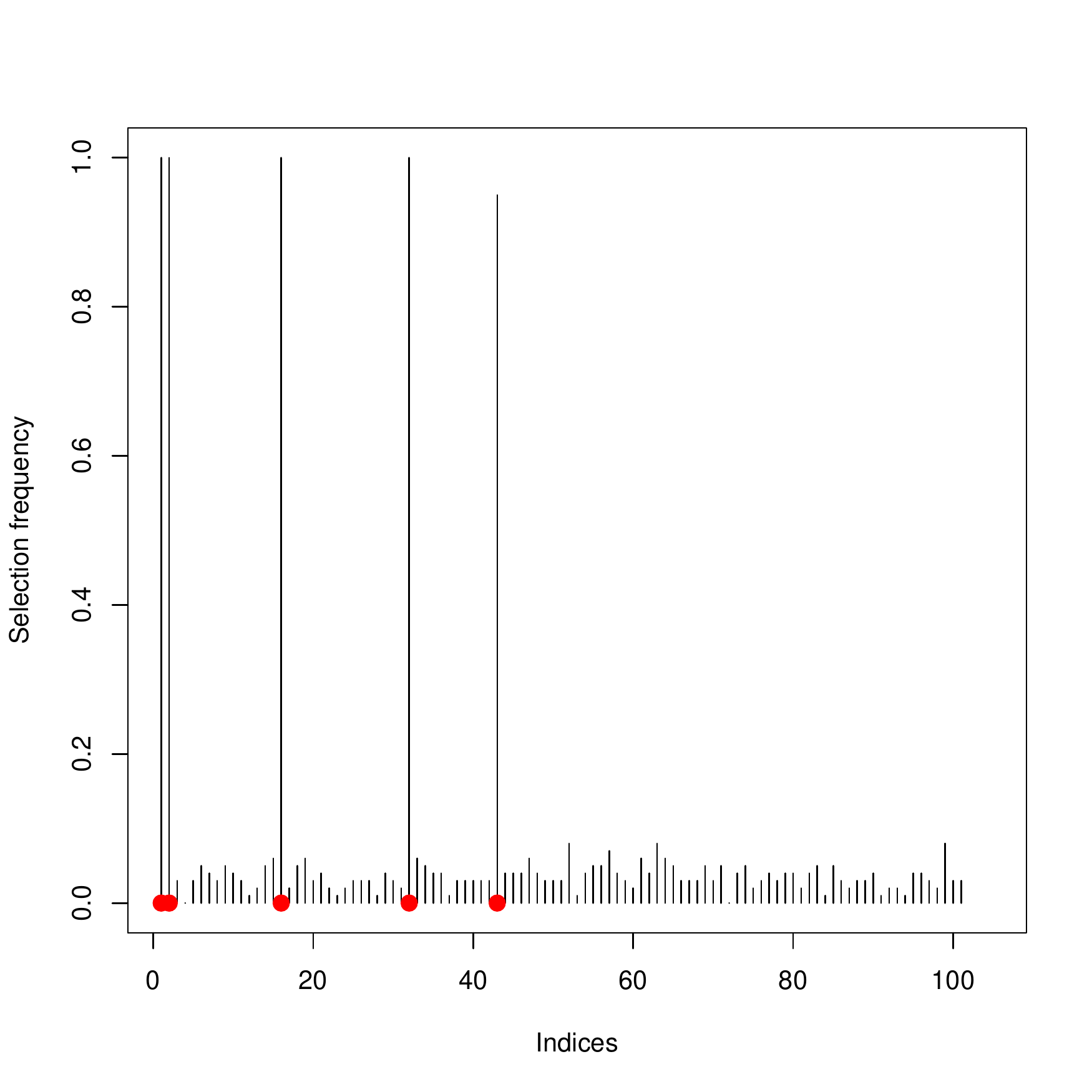}
\includegraphics[scale=0.35]{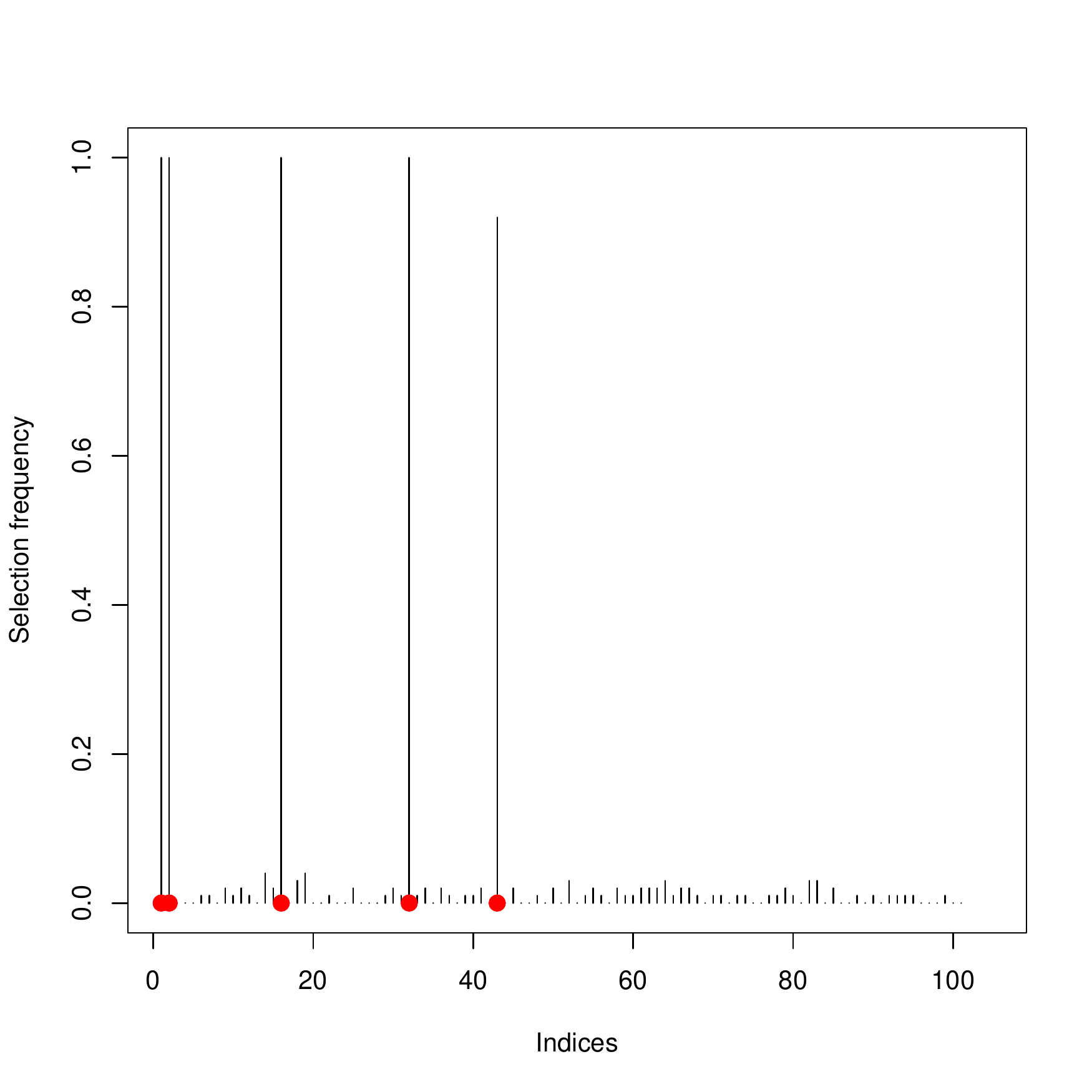}
\includegraphics[scale=0.35]{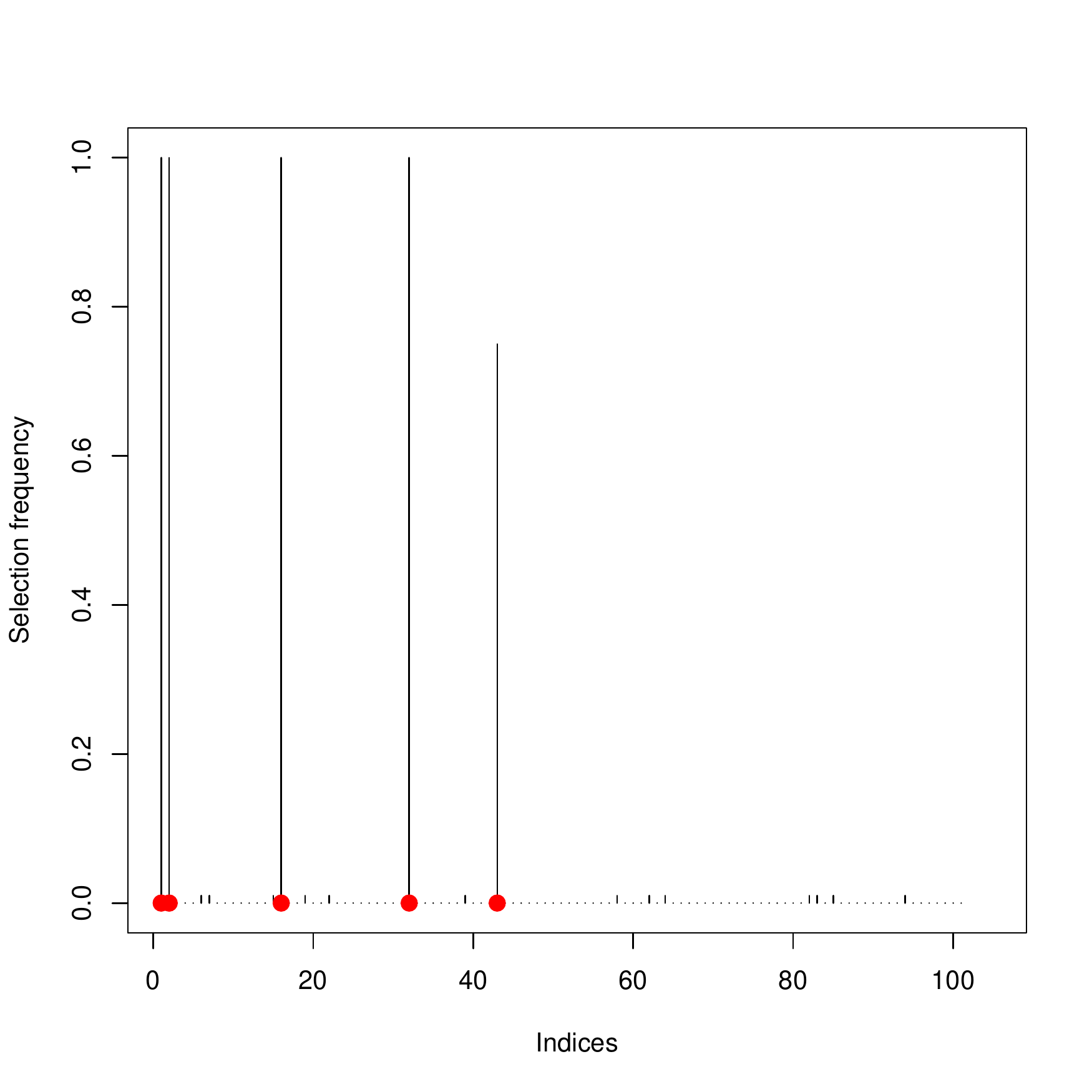}
\includegraphics[scale=0.35]{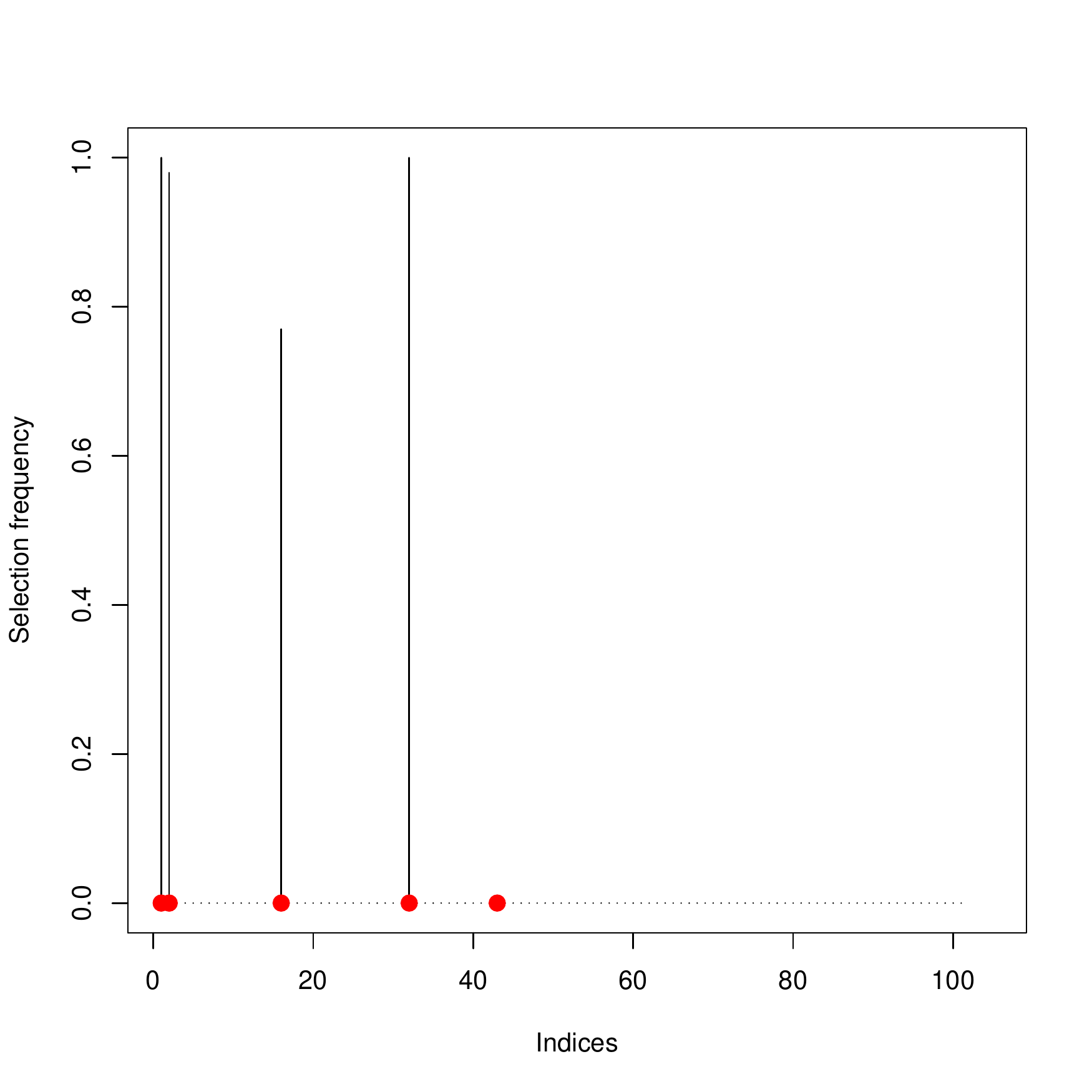}
\caption{Means of the selection frequencies of the indices of the final estimator of $\boldsymbol{\beta}^\star$ based on 100 replications 
for different thresholds: 0.7 (top left), 0.8 (top right), 0.9 (bottom left) and
1 (bottom right). 
The positions of the non null values of $\boldsymbol{\beta}^\star$ are displayed with red plain circles.\label{fig:freq_barres}}
\end{figure}

\subsection{Numerical performance}

Figure \ref{fig:time} displays the means and standard errors of the computational times for our variable selection method. 
We can see from this figure that it takes only around 30 seconds to process observations $Y_1,\dots,Y_n$
satisfying (\ref{eq:Yt}), (\ref{eq:mut_Wt}) and (\ref{eq:Zt}) when $n=1000$, $p=100$, $q=3$ and when the number of replications used in the stability selection step described in Section 
\ref{sec:lambda} is equal to 1000.

\begin{figure}[!htbp]
\includegraphics[scale=0.3]{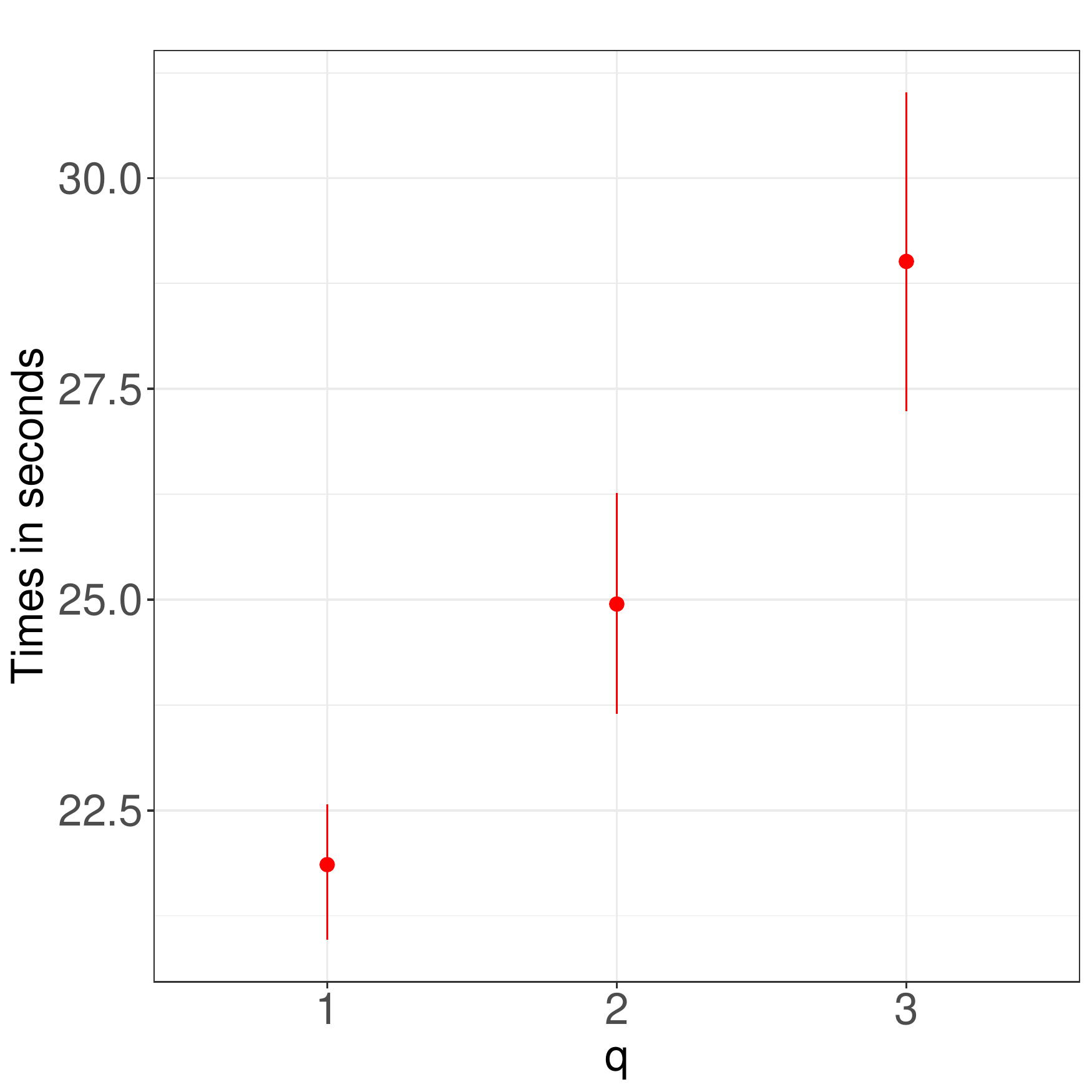}
\caption{Means and standard errors of the computational times in seconds for our variable selection approach in the case where $n=1000$, $p=100$, $q=1$, 2 and 3.\label{fig:time}}
\end{figure}

%\section{Conclusion}

\section{Proofs}\label{sec:proofs}

\subsection{\textcolor{black}{Computation of the first and second derivatives of $W_t$ defined in (\ref{eq:Wt})}}

The computations given below are similar to those provided in \cite{davis:dunsmuir:street:2005} but are specific to the parametrization 
$\boldsymbol{\delta}=(\boldsymbol{\beta}',\boldsymbol{\gamma}')$ considered in this paper.

\subsubsection{\textcolor{black}{Computation of the first derivatives of $W_t$ }}\label{subsub:first_derive}

By the definition of $W_t$ given in (\ref{eq:Wt}), we get
\begin{equation*}
\frac{\partial W_t}{\partial \boldsymbol{\delta}}(\boldsymbol{\delta})=\frac{\partial\boldsymbol{\beta}' x_t}{\partial \boldsymbol{\delta}}+\frac{\partial Z_t}{\partial \boldsymbol{\delta}}
(\boldsymbol{\delta}),
\end{equation*}
%}
where $\boldsymbol{\beta}$, $x_t$ and $Z_t$ are defined in (\ref{eq:Wt}). 
More precisely, for all $k\in\{0,\dots,p\}$, $\ell\in\{1,\dots,q\}$ and $t\in\{1,\dots,n\}$, by (\ref{eq:Et}),
\begin{align}\label{eq:gradW_beta}
\frac{\partial W_t}{\partial \beta_k}&=x_{t,k}+\frac{\partial Z_t}{\partial \beta_k}=x_{t,k}+\sum_{j=1}^{q\wedge (t-1)}\gamma_j\frac{\partial E_{t-j}}{\partial \beta_k}\nonumber\\
&=x_{t,k}-\sum_{j=1}^{q\wedge (t-1)}\gamma_j Y_{t-j}\frac{\partial W_{t-j}}{\partial \beta_k}\exp(-W_{t-j})=x_{t,k}-\sum_{j=1}^{q\wedge (t-1)}\gamma_j(1+E_{t-j})\frac{\partial W_{t-j}}{\partial \beta_k},\\
\frac{\partial W_t}{\partial \gamma_\ell}&=E_{t-\ell}+\sum_{j=1}^{q\wedge (t-1)} \gamma_j\frac{\partial E_{t-j}}{\partial\gamma_\ell}\nonumber\\\label{eq:gradW_gamma}
&=E_{t-\ell}-\sum_{j=1}^{q\wedge (t-1)}\gamma_j Y_{t-j}\frac{\partial W_{t-j}}{\partial \gamma_\ell}\exp(-W_{t-j})=E_{t-\ell}-\sum_{j=1}^{q\wedge (t-1)}\gamma_j(1+E_{t-j})\frac{\partial W_{t-j}}{\partial \gamma_\ell},
\end{align}
where we used that  $E_t=0,\; \forall t\leq 0$.

The first derivatives of $W_t$ are thus obtained from the following recursive expressions. For all $k\in\{0,\dots,p\}$ 
\begin{align*}
\frac{\partial W_1}{\partial \beta_k}&=x_{1,k},\\
\frac{\partial W_2}{\partial \beta_k}&=x_{2,k}-\gamma_1(1+E_{1})\frac{\partial W_{1}}{\partial \beta_k},
\end{align*}
where
\begin{equation}\label{eq:E1}
W_1=\boldsymbol{\beta}' x_1 \textrm{ and } E_1=Y_1\exp(-W_1)-1.
\end{equation}
Moreover,
\begin{equation*}
\frac{\partial W_3}{\partial \beta_k}=x_{3,k}-\gamma_1(1+E_{2})\frac{\partial W_{2}}{\partial \beta_k}-\gamma_2(1+E_{1})\frac{\partial W_{1}}{\partial \beta_k},
\end{equation*}
where
\begin{equation}\label{eq:E2}
W_2=\boldsymbol{\beta}' x_2  +\gamma_1 E_{1},\; E_2=Y_2\exp(-W_2)-1,
\end{equation}
and so on. In the same way, for all $\ell\in\{1,\dots,q\}$
\begin{align*}
\frac{\partial W_1}{\partial \gamma_\ell}&=0,\\
\frac{\partial W_2}{\partial \gamma_\ell}&=E_{2-\ell},\\
\frac{\partial W_3}{\partial \gamma_\ell}&=E_{3-\ell}-\gamma_1(1+E_{2})\frac{\partial W_{2}}{\partial \gamma_\ell}%=E_{3-\ell}-\gamma_1(1+E_{2})E_{2-\ell},
\end{align*}
and so on, where $E_t=0,\; \forall t\leq 0$ and $E_1$, $E_2$ are defined in (\ref{eq:E1}) and (\ref{eq:E2}), respectively.

\subsubsection{\textcolor{black}{Computation of the second derivatives of $W_t$}}\label{subsub:second_derive}

Using (\ref{eq:gradW_beta}) and (\ref{eq:gradW_gamma}), we get that for all $j,k\in\{0,\dots,p\}$, $\ell,m\in\{1,\dots,q\}$ and $t\in\{1,\dots,n\}$,
\begin{align*}
\frac{\partial^2 W_t}{\partial \beta_j\partial \beta_k}&=-\sum_{i=1}^{q\wedge (t-1)}\gamma_i(1+E_{t-i})\frac{\partial^2 W_{t-i}}{\partial \beta_j\partial \beta_k}
-\sum_{i=1}^{q\wedge (t-1)}\gamma_i\frac{\partial E_{t-i}}{\partial\beta_j}\frac{\partial W_{t-i}}{\partial \beta_k}\\
&=-\sum_{i=1}^{q\wedge (t-1)}\gamma_i(1+E_{t-i})\frac{\partial^2 W_{t-i}}{\partial \beta_j\partial \beta_k}
+\sum_{i=1}^{q\wedge (t-1)}\gamma_i(1+E_{t-i})\frac{\partial W_{t-i}}{\partial \beta_j}\frac{\partial W_{t-i}}{\partial \beta_k},\\
\frac{\partial^2 W_t}{\partial \beta_k\partial\gamma_\ell}&=-(1+E_{t-\ell})\frac{\partial W_{t-\ell}}{\partial \beta_k}
-\sum_{i=1}^{q\wedge (t-1)}\gamma_i\left\{\frac{\partial W_{t-i}}{\partial \beta_k}\frac{\partial E_{t-i}}{\partial\gamma_\ell}
                                                            +(1+E_{t-i})\frac{\partial^2 W_{t-i}}{\partial \beta_k\partial\gamma_\ell}\right\}\\
&=-(1+E_{t-\ell})\frac{\partial W_{t-\ell}}{\partial \beta_k}
-\sum_{i=1}^{q\wedge (t-1)}\gamma_i\left\{-(1+E_{t-i})\frac{\partial W_{t-i}}{\partial\beta_k}\frac{\partial W_{t-i}}{\partial \gamma_\ell}
                                                            +(1+E_{t-i})\frac{\partial^2 W_{t-i}}{\partial \beta_k\partial\gamma_\ell}\right\},\\
\frac{\partial^2 W_t}{\partial \gamma_\ell\partial\gamma_m}&=\frac{\partial E_{t-\ell}}{\partial \gamma_m}
-(1+E_{t-m})\frac{\partial W_{t-m}}{\partial \gamma_\ell} 
-\sum_{i=1}^{q\wedge (t-1)}\gamma_i\left\{\frac{\partial W_{t-i}}{\partial \gamma_\ell} \frac{\partial E_{t-i}}{\partial \gamma_m}
+(1+E_{t-i})\frac{\partial^2 W_{t-i}}{\partial \gamma_\ell\partial \gamma_m}\right\}\\
&=-(1+E_{t-\ell})\frac{\partial W_{t-\ell}}{\partial \gamma_m}-(1+E_{t-m})\frac{\partial W_{t-m}}{\partial \gamma_\ell} \\
&-\sum_{i=1}^{q\wedge (t-1)}\gamma_i\left\{-(1+E_{t-i})\frac{\partial W_{t-i}}{\partial \gamma_\ell}\frac{\partial W_{t-i}}{\partial \gamma_m}
+(1+E_{t-i})\frac{\partial^2 W_{t-i}}{\partial \gamma_\ell\partial \gamma_m}\right\}.\\
\end{align*}

To compute the second derivatives of $W_t$, we shall use the following recursive expressions for all $j,k\in\{0,\dots,p\}$
\begin{align*}
\frac{\partial^2 W_1}{\partial \beta_j\partial \beta_k}&=0,\\
\frac{\partial^2 W_2}{\partial \beta_j\partial \beta_k}&=\gamma_1(1+E_1)x_{1,j}x_{1,k},
\end{align*}
where $E_1$ is defined in (\ref{eq:E1}) and so on. Moreover, for all $k\in\{0,\dots,p\}$ and $\ell\in\{1,\dots,q\}$
\begin{align*}
\frac{\partial^2 W_1}{\partial \beta_k\partial\gamma_\ell}&=0,\\
\frac{\partial^2 W_2}{\partial \beta_k\partial\gamma_\ell}&=-(1+E_{2-\ell})\frac{\partial W_{2-\ell}}{\partial \beta_k},
\end{align*}
where $E_t=0$ for all $t\leq 0$ and the first derivatives of $W_t$ are computed in (\ref{eq:gradW_beta}).
Note also that
\begin{align*}
\frac{\partial^2 W_1}{\partial \gamma_\ell\partial\gamma_m}&=0,\\
\frac{\partial^2 W_2}{\partial \gamma_\ell\partial\gamma_m}&=0
\end{align*}
and so on.

\subsection{Computational details for obtaining Criterion (\ref{eq:beta_hat})}\label{sub:var_sec}

By \eqref{eq:Ltilde},
\begin{align*}
\widetilde{L}(\boldsymbol{\beta})=\widetilde{L}(\widetilde{\boldsymbol{\beta}})+\frac{\partial L}{\partial \boldsymbol{\beta}}(\widetilde{\boldsymbol{\beta}},\widehat{\boldsymbol{\gamma}})
U(\boldsymbol{\nu}-\widetilde{\boldsymbol{\nu}})-\frac12 (\boldsymbol{\nu}-\widetilde{\boldsymbol{\nu}})' \Lambda (\boldsymbol{\nu}-\widetilde{\boldsymbol{\nu}}),
\end{align*}
%\textcolor{blue}{OK pour le "-" dans l'approximation de Taylor (à rediscuter) ; il faut aussi le mettre dans la définition de $\tilde{L}(\boldsymbol{\beta})$ ci-dessus}
where $\boldsymbol{\nu}-\widetilde{\boldsymbol{\nu}}=U'(\boldsymbol{\beta}-\widetilde{\boldsymbol{\beta}})$.
Hence,
\begin{align*}
\widetilde{L}(\boldsymbol{\beta})&=\widetilde{L}(\widetilde{\boldsymbol{\beta}})+\sum_{k=0}^p
\left(\frac{\partial L}{\partial \boldsymbol{\beta}}(\widetilde{\boldsymbol{\beta}},\widehat{\boldsymbol{\gamma}}) U\right)_k (\nu_k-\widetilde{\nu}_k)
-\frac12\sum_{k=0}^p\lambda_k (\nu_k-\widetilde{\nu}_k)^2\\
&=\widetilde{L}(\widetilde{\boldsymbol{\beta}})-\frac12\sum_{k=0}^p\lambda_k\left(\nu_k-\widetilde{\nu}_k-\frac{1}{\lambda_k}
\left(\frac{\partial L}{\partial \boldsymbol{\beta}}(\widetilde{\boldsymbol{\beta}},\widehat{\boldsymbol{\gamma}}) U\right)_k\right)^2
+\sum_{k=0}^p\frac{1}{2\lambda_k}\left(\frac{\partial L}{\partial \boldsymbol{\beta}}(\widetilde{\boldsymbol{\beta}},\widehat{\boldsymbol{\gamma}}) U\right)_k^2,
\end{align*}
where the $\lambda_k$'s are the diagonal terms of $\Lambda$.

Since the only term depending on $\boldsymbol{\beta}$ is the second one in the last expression of $\widetilde{L}(\boldsymbol{\beta})$,
we define $\widetilde{L}_Q(\boldsymbol{\beta})$ appearing in Criterion (\ref{eq:beta_hat}) as follows:
\begin{eqnarray*}
-\widetilde{L}_Q(\boldsymbol{\beta})&=&\frac12\sum_{k=0}^p\lambda_k\left(\nu_k-\widetilde{\nu}_k+\frac{1}{\lambda_k}
\left(\frac{\partial L}{\partial \boldsymbol{\beta}}(\widetilde{\boldsymbol{\beta}},\widehat{\boldsymbol{\gamma}}) U\right)_k\right)^2\\
&=&\frac12 \left\|\Lambda^{1/2}\left(\boldsymbol{\nu}-\widetilde{\boldsymbol{\nu}}+\Lambda^{-1} \left(\frac{\partial L}{\partial \boldsymbol{\beta}}(\widetilde{\boldsymbol{\beta}},\widehat{\boldsymbol{\gamma}}) U\right)'
\right)\right\|_2^2\\
&=&\frac12 \left\|\Lambda^{1/2}U'(\boldsymbol{\beta}-\widetilde{\boldsymbol{\beta}})+\Lambda^{-1/2} U' \left(\frac{\partial L}{\partial \boldsymbol{\beta}}(\widetilde{\boldsymbol{\beta}},\widehat{\boldsymbol{\gamma}})\right)'
\right\|_2^2\\
&=&\frac12 \left\|\Lambda^{1/2}U'(\widetilde{\boldsymbol{\beta}}-\boldsymbol{\beta})-\Lambda^{-1/2} U' \left(\frac{\partial L}{\partial \boldsymbol{\beta}}(\widetilde{\boldsymbol{\beta}},\widehat{\boldsymbol{\gamma}})\right)'\right\|_2^2\\
&=&\frac12\|\mathcal{Y}-\mathcal{X}\boldsymbol{\beta}\|_2^2,
\end{eqnarray*}
where
\begin{equation}\label{eq:def_Y_X}
\mathcal{Y}=\Lambda^{1/2}U'\widetilde{\boldsymbol{\beta}}
-\Lambda^{-1/2}U'\left(\frac{\partial L}{\partial \boldsymbol{\beta}}(\widetilde{\boldsymbol{\beta}},\widehat{\boldsymbol{\gamma}})\right)' ,\;  \mathcal{X}=\Lambda^{1/2}U'.
\end{equation}

\subsection{Proofs of Propositions \ref{prop1}, \ref{prop2} and \ref{prop3} and of Lemma \ref{lem:aperiodic_doeblin}}

This section contains the proofs of Propositions \ref{prop1}, \ref{prop2} and \ref{prop3}.
% and of Lemma \ref{lem:aperiodic_doeblin} which is used in the proof of Proposition \ref{prop1}.

\subsubsection{\textcolor{black}{Proof of Proposition \ref{prop1}}}

\textcolor{black}{We first establish the following lemma for proving Proposition \ref{prop1}.}

\begin{lemma}\label{lem:aperiodic_doeblin}
$(W_t^\star)$ is an aperiodic Markov process satisfying Doeblin's condition.
\end{lemma}

\begin{proof}[Proof of Lemma \ref{lem:aperiodic_doeblin}]
By (\ref{eq:mut_simple}) and (\ref{eq:Zt}), we observe that:
\begin{equation}\label{eq:Wtstar}
W_t^\star=(\beta_0^\star-\gamma_1^\star)+\gamma_1^\star Y_{t-1}\exp(-W_{t-1}^\star).
\end{equation}
Thus, $\mathcal{F}_{t-2}=\mathcal{F}_{t-1}^{W^\star}:=\sigma(W_s,s\leq t-1)$.
By (\ref{eq:Yt}), the distribution of $Y_{t-1}$ conditionally to $\mathcal{F}_{t-2}$ is $\mathcal{P}(\exp(W_{t-1}^\star))$.
Hence, the distribution of $W_t^\star$ conditionally to $\mathcal{F}_{t-1}^{W^\star}$
is the same as distribution of $W_t^\star$ conditionally to $W_{t-1}^\star$, which means that $(W_t^\star)$ has the Markov property.

Let us now prove that $(W_t^\star)$ is strongly aperiodic which implies that it is aperiodic.
$$
\mathbb{P}(W^\star_{t}=\beta_0^\star-\gamma_1^\star | W^\star_{t-1}=\beta_0^\star-\gamma_1^\star)
=\mathbb{P}(Y_{t-1}=0 | W^\star_{t-1}=\beta_0^\star-\gamma_1^\star)=\exp(-\exp(\beta_0^\star-\gamma_1^\star))>0,
$$
where the first equality comes from (\ref{eq:Wtstar}) and the last equality comes from (\ref{eq:Yt}) since $\mathcal{F}_{t-2}=\mathcal{F}_{t-1}^{W^\star}$.

To prove that $(W_t^\star)$ satisfies Doeblin's condition namely that there exists a probability measure $\nu$ with the property that, for some $m\geq 1$, $\varepsilon>0$
and $\delta >0$,
\begin{equation}\label{eq:doeblin}
\nu(B)>\varepsilon \Longrightarrow \mathbb{P}(W_{t+m-1}\in B,W_{t+m-2}\in B\dots,W_{t+1}\in B,W_{t}\in B|W_{t-1}=x)\geq\delta,
\end{equation}
for all $x$ in the state space $X$ of $W_t^\star$ and $B$ in the Borel sets of $X$, we refer the reader to the proof of Proposition 2 in \cite{davis:dunsmuir:streett:2003}.

% Let us first focus on the case where $\gamma_1^\star<0$. Then, by (\ref{eq:mut_simple}) and (\ref{eq:Zt}),
% $$
% W_t^\star=(\beta_0^\star-\gamma_1^\star)+\gamma_1^\star Y_{t-1}\exp(-W_{t-1}^\star)\leq \beta_0^\star-\gamma_1^\star.
% $$
% Let $\nu$ be a measure having a unit mass in $\beta_0^\star-\gamma_1^\star$ and $B$ the Borel sets such  $\beta_0^\star-\gamma_1^\star\in B$.
% Then, for all $x\leq \beta_0^\star-\gamma_1^\star$, 
% $$
% \mathbb{P}(W_{t+1}\in B|W_{t-1}=x)\geq\mathbb{P}(W_{t+1}\in B|W_{t-1}=x)\geq \mathbb{P}(W_{t+1}=\beta_0^\star-\gamma_1^\star|W_{t-1}=x)
% $$
\end{proof}

\begin{proof}[Proof of Proposition \ref{prop1}]
For proving Proposition \ref{prop1}, we shall use Theorems 1.3.3 and 1.3.5 of
\cite{taniguchibook:2012}. In order to apply these theorems it is enough to prove that $(W_t^\star)$ is a strictly stationary and ergodic process
since $Y_t W_t(\beta_0^\star,\gamma_1)-\exp(W_t(\beta_0^\star,\gamma_1))$ is a measurable function of $W_{t+1}^\star,W_t^\star,\dots,W_2^\star$. 
Note that the latter fact comes from (\ref{eq:mut_simple}) and (\ref{eq:Zt}) for $Y_t$ and from (\ref{eq:Wt}) with $q=1$ and $p=0$ for $W_t$.

In order to prove that $(W_t^\star)$ is a strictly stationary and ergodic process, we have first to prove that $(W_t^\star)$ is an aperiodic Markov process satisfying Doeblin's condition, 
see Lemma \ref{lem:aperiodic_doeblin}.

The statement of Lemma \ref{lem:aperiodic_doeblin} corresponds to Assertion (iv) of Theorem 16.0.2 of \cite{meyn:tweedie}  which is equivalent to Assertion (i) of this theorem, 
and implies that $(W_t^\star)$ is uniformly ergodic.

Hence, by Definition (16.6) of uniform ergodicity given in \cite{meyn:tweedie}, there exists a unique stationary invariant measure for $(W_t^\star)$, see also the paragraph below Equation (1.3) 
of \cite{Sandric:2017} for an additional justification. Combining that existence of a unique stationary invariant measure for $(W_t^\star)$ with the following arguments shows that $(W_t^\star)$ 
is a strictly stationary process and also an ergodic Markov process.

By Theorem 3.6.3, Corollary 3.6.1 and Definition 3.6.6 of \cite{stout:1974}, if the process $(W_t^\star)$ is started with its unique stationary invariant distribution, $(W_t^\star)$ is a strictly stationary process.

By Definition 3.6.8 of  \cite{stout:1974}, the existence of a unique stationary invariant measure for $(W_t^\star)$ 
means that $(W_t^\star)$ is an ergodic Markov process, see also the paragraph below (b) \cite[p. 717]{Sandric:2017}.

Finally, by Theorem 3.6.5 of  \cite{stout:1974}, since $(W_t^\star)$ is an ergodic Markov process and a strictly stationary process, 
$(W_t^\star)$ is an ergodic and strictly stationary process in the sense of the assumption of Theorem 1.3.5 of \cite{taniguchibook:2012}.
\end{proof}

\subsubsection{\textcolor{black}{Proof of Proposition \ref{prop2}}}

Note that for all $\gamma_1$,
\begin{align*}
\mathcal{L}(\gamma_1)&=\mathbb{E}\left[Y_3 W_3(\beta_0^\star,\gamma_1)-\exp(W_3(\beta_0^\star,\gamma_1))\right]
=\mathbb{E}\left[\mathbb{E}\left[Y_3 W_3(\beta_0^\star,\gamma_1)-\exp(W_3(\beta_0^\star,\gamma_1))|\mathcal{F}_2\right]\right]\\
&=\mathbb{E}\left[\exp(W_3^\star) W_3(\beta_0^\star,\gamma_1)-\exp(W_3(\beta_0^\star,\gamma_1))\right]\\
&=\mathbb{E}\left[\exp(W_3^\star) \left(W_3(\beta_0^\star,\gamma_1)-W_3^\star+W_3^\star-\exp(W_3(\beta_0^\star,\gamma_1)-W_3^\star)\right)\right]\\
&\leq \mathbb{E}\left[\exp(W_3^\star) \left(W_3^\star-1\right)\right]=\mathcal{L}(\gamma_1^\star),
\end{align*}
where the inequality comes from the following inequality $x-\exp(x)\leq -1$, for all $x\in\mathbb{R}$. This inequality is an equality only when 
$x=0$ which means that $\gamma_1=\gamma_1^\star$.

\subsubsection{\textcolor{black}{Proof of Proposition \ref{prop3}}}

%\begin{proof}[Proof of Proposition \ref{prop3}]
The proof of this proposition comes from Proposition \ref{prop1} and the stochastic equicontinuity of $n^{-1}L(\beta_0^\star,\gamma_1)$. Thus, it is enough to prove that
there exists a positive $\delta$ such that
$$
\sup_{|\gamma_1-\gamma_2|\leq\delta}\left|\frac{L(\beta_0^\star,\gamma_1)}{n}-\frac{L(\beta_0^\star,\gamma_2)}{n}\right|\stackrel{p}{\longrightarrow}0, \textrm{ as $n$ tends to infinity.}
$$
Observe that, by (\ref{eq:L:beta_0}),
\begin{align*}
\left|\frac{L(\beta_0^\star,\gamma_1)}{n}-\frac{L(\beta_0^\star,\gamma_2)}{n}\right|
&\leq\frac1n\sum_{t=1}^n Y_t\left|W_t(\beta_0^\star,\gamma_1)-W_t(\beta_0^\star,\gamma_2)\right|\\
&+\frac1n\sum_{t=1}^n\left|\exp\left(W_t(\beta_0^\star,\gamma_1)\right)-\exp\left(W_t(\beta_0^\star,\gamma_2)\right)\right|.
\end{align*}
Let us first focus on bounding the following expression for $t\geq 2$ (since $W_1(\beta_0^\star,\gamma)=\beta_0^\star$, for all $\gamma$). By (\ref{eq:W_Z})
\begin{align*}
&\left|W_t(\beta_0^\star,\gamma_1)-W_t(\beta_0^\star,\gamma_2)\right|=\left|Z_t(\gamma_1)-Z_t(\gamma_2)\right|=\left|\gamma_1 E_{t-1}(\gamma_1)-\gamma_2 E_{t-1}(\gamma_2)\right|\\
&=\left|\gamma_1 \left[Y_{t-1}\exp(-W_{t-1}(\beta_0^\star,\gamma_1))-1\right]-\gamma_2 \left[Y_{t-1}\exp(-W_{t-1}(\beta_0^\star,\gamma_2))-1\right]\right|\\
&=\left|Y_{t-1}\textrm{e}^{-\beta_0^\star}\left[\gamma_1\exp(-Z_{t-1}(\gamma_1))-\gamma_2\exp(-Z_{t-1}(\gamma_2))\right]+\gamma_2-\gamma_1\right|\\
&\leq Y_{t-1}\textrm{e}^{-\beta_0^\star}\left[\left|\gamma_1-\gamma_2\right|\exp(-Z_{t-1}(\gamma_1))+|\gamma_2|\left|\exp(-Z_{t-1}(\gamma_1))-\exp(-Z_{t-1}(\gamma_2))\right|\right]
+\left|\gamma_2-\gamma_1\right|\\
&\leq Y_{t-1}\textrm{e}^{-\beta_0^\star}\left|\gamma_1-\gamma_2\right|\exp(-Z_{t-1}(\gamma_1))\\
&+Y_{t-1}\textrm{e}^{-\beta_0^\star}|\gamma_2|\exp(-Z_{t-1}(\gamma_1))\left|Z_{t-1}(\gamma_1)-Z_{t-1}(\gamma_2)\right|
\exp(|Z_{t-1}(\gamma_1)-Z_{t-1}(\gamma_2)|)\\
&+\left|\gamma_2-\gamma_1\right|,
\end{align*}
where we used in the last inequality that for all $x$ and $y$ in $\mathbb{R}$,
\begin{equation}\label{eq:exp_x_y}
|\textrm{e}^x-\textrm{e}^y|=\textrm{e}^x|1-\textrm{e}^{y-x}|\leq \textrm{e}^x |y-x|\textrm{e}^{|y-x|}.
\end{equation}
Observing that
\begin{equation}\label{eq:exp_Zt}
\exp(-Z_{t}(\gamma_1))=\exp\left(-\gamma_1\left[Y_{t-1}\textrm{e}^{-\beta_0^\star}\exp(-Z_{t-1}(\gamma_1))-1\right]\right),
\end{equation}
and $|Z_{2}(\gamma_1)-Z_{2}(\gamma_2)|\leq\delta[Y_1\textrm{e}^{-\beta_0^\star}+1]$ we get, for $\gamma_1$ and $\gamma_2$ such that $|\gamma_1-\gamma_2|\leq\delta$, that
\begin{equation}\label{eq:diff_W}
\left|W_t(\beta_0^\star,\gamma_1)-W_t(\beta_0^\star,\gamma_2)\right|
\leq\delta\; F(Y_{t-1},Y_{t-2},\dots,Y_1),
\end{equation}
where $F$ is a measurable function.
By (\ref{eq:exp_x_y}),
\begin{align*}
&\left|\exp\left(W_t(\beta_0^\star,\gamma_1)\right)-\exp\left(W_t(\beta_0^\star,\gamma_2)\right)\right|\\
&\leq \exp\left(W_t(\beta_0^\star,\gamma_1)\right)\left|W_t(\beta_0^\star,\gamma_1)-W_t(\beta_0^\star,\gamma_2)\right|\exp\left(\left|W_t(\beta_0^\star,\gamma_1)-W_t(\beta_0^\star,\gamma_2)\right|\right)\\
&\leq \delta G(Y_{t-1},Y_{t-2},\dots,Y_1)
\end{align*}
where the last inequality comes from (\ref{eq:diff_W}), (\ref{eq:exp_Zt}) and (\ref{eq:W_Z}) and where $G$ is a measurable function.
Thus, we get that 
$$
\left|\frac{L(\beta_0^\star,\gamma_1)}{n}-\frac{L(\beta_0^\star,\gamma_2)}{n}\right|\leq\frac{\delta}{n}\sum_{t=1}^n H(Y_t,Y_{t-1},\dots,Y_1),
$$
which gives the result by using similar arguments as those given in the proof of Proposition \ref{prop1} namely that $(Y_t)$ is strictly stationary and ergodic.
By Theorem 1.3.3 of \cite{taniguchibook:2012}, $H(Y_t,Y_{t-1},\dots,Y_1)$ is strictly stationary and ergodic since $(Y_t)$ has these properties. Thus, $\mathbb{E}[|H(Y_t,Y_{t-1},\dots,Y_1)|]<\infty$, which
concludes the proof by Theorem 1.3.5 of \cite{taniguchibook:2012}.

% it is enough to focus on bounding
% $$
% \left|W_t(\beta_0^\star,\gamma_1)-W_t(\beta_0^\star,\gamma_2)\right|=\left|Z_t(\gamma_1)-Z_t(\gamma_2)\right|
% $$
% and
% $$
% \exp\left(|W_t(\beta_0^\star,\gamma_1)|\right)\leq\textrm{e}^{|\beta_0^\star|}\exp\left(|Z_t(\gamma_1)|\right).
% $$

%\end{proof}

\bibliographystyle{chicago}
\bibliography{biblio}

\end{document}